\documentclass[11pt]{article}

\usepackage{amsmath, amssymb, amsthm}
\usepackage{dsfont}
\usepackage{graphicx}
\usepackage{tikz}
\usepackage{float}
\usepackage{enumerate}
\usepackage{hyperref}

\usetikzlibrary{calc}
\usetikzlibrary{shapes, positioning}
\usetikzlibrary{backgrounds}

\allowdisplaybreaks
\expandafter\let\expandafter\oldproof\csname\string\proof\endcsname
\let\oldendproof\endproof
\renewenvironment{proof}[1][\proofname]{%
	\oldproof[\bf #1]%
}{\oldendproof}

\allowdisplaybreaks

\theoremstyle{plain}
\newtheorem{lemma}{Lemma}[section]
\newtheorem{theorem}[lemma]{Theorem}
\newtheorem{claim}[lemma]{Claim}

\newtheorem{definition}[lemma]{Definition}

\newtheorem{problem}[lemma]{Problem}

\newtheorem{construction}[lemma]{Construction}

\newtheorem{setting}[lemma]{Setting}
\newtheorem*{claim*}{Claim}

\RequirePackage[normalem]{ulem} 
\RequirePackage{color}\definecolor{RED}{rgb}{1,0,0}\definecolor{BLUE}{rgb}{0,0,1} 

\begin{document}

\title{Multicolor $K_r$-Tilings with High Discrepancy}

\author{
Henry Chan\thanks{Department of Mathematics, University of Toronto, Canada. \emph{Email}: \href{hokhin.chan@mail.utoronto.ca}{hokhin.chan@mail.utoronto.ca}. Supported by an Undergraduate Student Research Award.}
\and 
Daniel Cheng\thanks{Department of Mathematics, University of Toronto, Canada. \emph{Email}: \href{danielx.cheng@mail.utoronto.ca}{danielx.cheng@mail.utoronto.ca}. Supported by a University of Toronto Research Excellence Award.}
\and 
Lior Gishboliner\thanks{Department of Mathematics, University of Toronto, Canada.
Supported in part by the NSERC Discovery Grant ``Problems in Extremal and Probabilistic Combinatorics".
\emph{Email}: \href{mailto:lior.gishboliner@utoronto.ca}{\tt lior.gishboliner@utoronto.ca}.}
\and 
Xiangyu Li\thanks{Department of Mathematics, University of Toronto, Canada. \emph{Email}: \href{xiangyuu.li@mail.utoronto.ca}{xiangyuu.li@mail.utoronto.ca}.}
}

\date{}

\maketitle

\begin{abstract}
    We study the minimum degree threshold $\delta_{r,q}$ guaranteeing the existence of $K_r$-tilings of high discrepancy in any $q$-edge-coloring. Balogh, Csaba, Pluhár and Treglown handled the 2-color case, proving that $\delta_{r,2} = \frac{r}{r+1}$ for all $r \geq 3$. Here we determine $\delta_{r,q}$ for all $q$ large enough, namely $q \geq \binom{r}{2}$. For example, we show that for $r \geq 4$, $\delta_{r,q} = \frac{r}{r+1}$ for $\binom{r}{2} \leq q \leq \binom{r+1}{2}$ and $\delta_{r,q} = \frac{r-1}{r}$ for 
    $q \geq \binom{r+1}{2}+2$. Thus, $\delta_{r,q}$ has a phase transition at $q = \binom{r+1}{2}$, where it drops from $\frac{r}{r+1}$ and then stabilizes at the existence threshold $\frac{r-1}{r}$. We also show that $\delta_{r,q} \leq \frac{r}{r+1}$ for all $r,q$, supplementing and giving a new proof for the result of Balogh, Csaba, Pluhár and Treglown.
\end{abstract}

	\section{Introduction}


    In recent years there has been considerable interest in discrepancy problems for graphs. 
    The setting in such problems is as follows: One is a given a graph $G$ and a family $\mathcal{F}$ of subgraphs of $G$. The goal is to show that in every $q$-coloring of $E(G)$ (for some fixed number of colors $q$), there exists some $F \in \mathcal{F}$ such that the coloring of $F$ is unbalanced, namely, one of the colors appears on significantly more than $e(F)/q$ edges. Formally, we consider graphs $G$ with $q$-edge-colorings $f : E(G) \rightarrow [q]$. The pair $(G,f)$ is called a {\em $q$-edge-colored graph}. In this multicolored setting, discrepancy is defined \nolinebreak as \nolinebreak follows:
    \begin{definition}[Discrepancy]
    Let $(F,f)$ be a $q$-edge-colored graphs. The discrepancy of $(F,f)$ is the largest $t$ such that there exists a color $c \in [q]$ which appears on at least $\frac{e(F)+t}{q}$ of the edges of $F$.  
    \end{definition}

    The goal is then, given a graph $G$ and a family of subgraphs $\mathcal{F}$ of $G$, to find conditions on $G$ guaranteeing that for every $q$-edge-coloring of $G$, there exists a graph $F \in \mathcal{F}$ with high discrepancy. Some early results of this type include the work of Erd\H{o}s and Spencer \cite{ES:72}, who studied the case that $\mathcal{F}$ is the set of cliques, and Erd\H{o}s, F\"uredi, Loebl and S\'os \cite{EFLS:95}, who studied the case where $\mathcal{F}$ is the set of copies of a give spanning tree. In both of these works, $G$ is taken to be $K_n$. 

    The works \cite{BCJP:20,BCPT:21} initiated the study of discrepancy problems for general graphs $G$. By now, there are numerous results of this type, for a variety of natural types of subgraphs: perfect matchings and Hamilton cycles~\cite{BCJP:20, BCL, FHLT:21,GKM_Hamilton, GKM_trees}, spanning trees~\cite{GKM_trees,HLMP}, $H$-factors~\cite{BCPT:21,BCG:23}, powers of Hamilton cycles~\cite{Bradac:22}, 1-factorizations \cite{AHIL}, and general bounded-degree graphs \cite{BPPR,HLMP}. See also \cite{BTZ-G:24,GGS_Hamilton,GGS_STS,HLMPSTZ24+,LMX24+} for similar results for uniform hypergraphs.

    In this paper we consider the problem of finding $K_r$-tilings with large discrepancy in $q$-edge-colorings of graphs of large minimum degree. First, let us recall the definition of an $H$-tiling.

    \begin{definition}
        An {\em $H$-tiling} of a graph $G$ is a collection of vertex-disjoint copies of $H$ which partition $V(G)$. 
    \end{definition}
    Note that $H$-tilings are sometimes called $H$-factors, perfect $H$-packings or perfect $H$-matchings. 
    The following theorem is a fundamental result in extremal graph theory, determining the minimum degree threshold which guarantees the existence of $K_r$-tilings.
    \begin{theorem}
        [Hajnal-Szemerédi~\cite{HS:70}]
        Let $r \geq 2$. Let $G$ be an $n$-vertex with $n$ divisible by $r$ and with minimum degree 
        $ \delta(G) \geq \frac{r-1}{r}n$. Then $G$ contains a $K_r$-tiling.
        \label{thm:hajnal sze}
    \end{theorem}
    Note that the bound above is tight; that is, there exists graphs $G$ with 
    $\delta(G) = \frac{r-1}{r}n - 1$ and no $K_r$-tiling. Indeed, consider the complete $r$-partite graph with parts of sizes $\frac{n}{r}-1,\frac{n}{r}+1,\frac{n}{r},\dots,\frac{n}{r}$ (assuming that $r$ divides $n$). This graph has minimum degree $\frac{n}{r}-1$, but no $K_r$-tiling (because every copy of $K_r$ must intersect all parts, but the smallest part has size $\frac{n}{r}-1$, while a $K_r$-tiling has size \nolinebreak $\frac{n}{r})$.
    
    Balogh, Csaba, Pluhár and Treglown \cite{BCPT:21} studied the existence of $K_r$-tilings with high discrepancy in 2-edge-colored-graph. 
    They showed that for $r \geq 3$, in any 2-edge-coloring of an $n$-vertex graph $G$ with $\delta(G) \geq (\frac{r}{r+1} + \gamma)n$ and $n$ divisible by $r$, there is a $K_r$-tiling with discrepancy $\Omega(n)$.\footnote{The case $r=2$ is different: The threshold for a perfect matching with high discrepancy is $\delta = \frac{3}{4}$; see \cite{BCJP:20} and also \cite{FHLT:21,GKM_trees} for the multicolor case.} 
    They also showed that the constant $\frac{r}{r+1}$ is best possible.
    The main objective of this paper is to generalize this work to the multicolored case. Namely, we consider the following minimum degree thresholds:
    \begin{definition}[$\delta_{r,q}$]\label{def:delta threshold}
        For $r \geq 3$ and $q \geq 2$, let $\delta_{r,q}$ be the infimum $\delta > 0$ such that there exists an $\zeta > 0$, so that for any large enough $n$, for any $n$-vertex graph $G$ with $r \mid n$ and $\delta(G) \geq \delta n$, and for any $q$-coloring $f : E(G) \rightarrow [q]$, there is a $K_r$-tiling in $(G,f)$ with discrepancy at least $\zeta n$.
    \end{definition}

    The aforementioned result of Balogh et al.~\cite{BCPT:21} gives $\delta_{r,2} = \frac{r}{r+1}$ for all $r \geq 3$. As our first result, we show that $\frac{r}{r+1}$ is an upper bound for any number of colors $q$. 
    
    \begin{theorem}\label{thm:general upper bound}
        For all $r \geq 3$ and $q \geq 2$, $\delta_{r,q} \leq \frac{r}{r+1}$.
    \end{theorem}
    \noindent
    We note that our proof is somewhat different than that of \cite{BCPT:21}, thus giving a new proof for the case $q=2$.

    As our main result, we determine $\delta_{r,q}$ for all $q$ large enough, namely, for all $q \geq \binom{r}{2}$. Interestingly, the threshold $\delta_{r,q}$ has a phase transition at $q = \binom{r+1}{2}+1$, where it decreases from its maximal possible value of $\frac{r}{r+1}$. Furthermore, $\delta_{r,q}$ then stabilizes at its minimum possible value, namely, the existence threshold $\frac{r-1}{r}$. The precise result for $r \geq 4$ is as follows. 

    \begin{theorem}\label{thm:r=4}
        Let $r \geq 4$. Then
        $$
        \delta_{r,q} = 
        \begin{cases}
            \frac{r}{r+1} & \binom{r}{2} \leq q \leq \binom{r+1}{2}, \\
            \frac{r^2+1}{r^2 + r + 2}
            & q = \binom{r+1}{2} + 1, \\
            \frac{r-1}{r} & q \geq \binom{r+1}{2}+2.
        \end{cases}
        $$
    \end{theorem}

    For $r=3$, the result is slightly different, in that there are two values of $q$ (instead of just one) where $\delta_{r,q}$ lies strictly between $\frac{r}{r+1}$ and $\frac{r-1}{r}$.

    \begin{theorem}\label{thm:r=3}
        For $r=3$, $\delta_{3,q} = \frac{3}{4}$ for $2 \leq q \leq 6$, $\delta_{3,7} = \frac{5}{7}$, $\delta_{3,8} = \frac{11}{16}$, and $\delta_{3,q} = \frac{2}{3}$ for $q \geq 9$.
    \end{theorem}

    Finally, we consider the range of small $q$, i.e., $q < \binom{r}{2}$, and show that the upper bound of $\frac{r}{r+1}$ is tight for all $q$ in this range satisfying a certain arithmetic condition. 

    \begin{theorem}\label{thm:divisibility condition}
    Let $r \geq 3$ and $2 \leq q \leq \binom{r}{2}$. Write $\binom{r}{2} = a \cdot q + b$ where $b < q$. If 
    $r+b \geq q$ or $b = 0$ then $\delta_{r,q} = \frac{r}{r+1}$.
    \end{theorem}

    \noindent
    The main problem left open by this work is to determine $\delta_{r,q}$ in all remaining cases:
    \begin{problem}
        Determine $\delta_{r,q}$ for all $r \geq 3$ and $q \leq \binom{r}{2}$.
    \end{problem}

    \paragraph{Notation:} We use $|G|$ to denote the number of vertices of a graph $G$. For a graph $G$ and a vertex-set $K \subseteq V(G)$, we denote by $N(K)$ the common neighborhood of $K$, namely, 
    $$
    N(K) := \{v \in V(G) \mid vw \in E(G) \text{ for all } w \in K\}.
    $$
    We always assume that $r \geq 3$, unless stated otherwise.

    \section{Extremal Constructions}

    In this section, we present the constructions providing the required lower bounds of Theorems \ref{thm:r=4}, \ref{thm:r=3} and \ref{thm:divisibility condition}.
    \begin{construction}\label{construction 1}
        Let $r \geq 3$ and $q \ge \binom{r}{2}$. Let $\alpha_1, \dots, \alpha_q \in \mathbb{Q}_{\ge 0}$ such that $\alpha := \sum_{i=1}^q \alpha_i \leq 1$. Let $n \in \mathbb{N}$ be such that $\alpha_i n \in \mathbb{N}$ for all $i \in [q]$ and $\frac{(1 - \alpha)n}{r} \in \mathbb{N}$. 
        
        Let $G$ be the complete $(r+1)$-partite graph on $n$ vertices with partition sets $V_1, \dots, V_{r+1}$, where $|V_1| = \dots = |V_r| = \frac{(1-\alpha)n}{r}$ and $|V_{r+1}| = \alpha n$. Further, partition $V_{r+1}$ into disjoint sets $Y_1, \dots, Y_q$ such that $|Y_i| = \alpha_i n$ for each $i \in [q]$.
    
        The edge coloring of $G$ is defined as follows: fix a bijection $g : \binom{[r]}{2} \rightarrow \{1, \dots, \binom{r}{2}\}$. For all $1 \leq i < j \leq r$, color the edges in $E(V_i, V_j)$ with color $g(\{i, j\})$. Finally, for each $k \in [q]$, color all edges incident to $Y_k$ with color $k$.
    \end{construction}
    \begin{construction}\label{construction 2}
    Let $r \geq 3$ and $q \leq \binom{r}{2}$, and write $\binom{r}{2} = aq + b$ for integers $a, b$ with $0 \leq b < q$. Let $\alpha_1, \dots, \alpha_q \in \mathbb{Q}_{\ge 0}$ such that $\alpha := \sum_{i=1}^q \alpha_i \leq 1$. Let $n \in \mathbb{N}$ be such that $\alpha_i n \in \mathbb{N}$ for all $i \in [q]$ and $\frac{(1 - \alpha)n}{r} \in \mathbb{N}$. 
    
    Let $G$ be the complete $(r+1)$-partite graph on $n$ vertices with partition sets $V_1, \dots, V_{r+1}$, satisfying $|V_1| = \dots = |V_r| = \frac{(1-\alpha)n}{r}$ and $|V_{r+1}| = \alpha n$. Further, partition $V_{r+1}$ into disjoint sets $Y_1, \dots, Y_q$ such that $|Y_k| = \alpha_k n$ for each $k \in [q]$.

    The edge coloring of $G$ is defined as follows: fix a function $g : \binom{[r]}{2} \rightarrow [q]$ such that $|g^{-1}(k)| = a + 1$ for $1\leq k \leq b$ and $|g^{-1}(k)| = a$ for $b < k \leq q$. For all $1 \leq i < j \leq r$, color the edges in $E(V_i, V_j)$ with color $g(\{i, j\})$. Finally, for each $k \in [q]$, color all edges incident to $Y_k$ with color $k$. 
    \end{construction}
    Figure \ref{fig:construction examples} depicts two instances of Construction \ref{construction 1}.
    \begin{figure}[H]
        \centering
        \begin{tikzpicture}[
          even_ellipse/.style={draw=black, thick, ellipse, minimum width=3.2cm, minimum height=1.3cm, align=center},
          long_ellipse/.style={draw=black, thick, ellipse, minimum width=4cm, minimum height=1.3cm, align=center},
          short_ellipse/.style={draw=black, thick, ellipse, minimum width=2cm, minimum height=1.3cm, align=center},
          dotted_circle/.style={circle, draw=black, dotted, thick, minimum size=0.6cm, inner sep=0pt},
          base_edge/.style={very thin, opacity=0.5, line cap=round},
          special_edge/.style={thin, opacity=0.8, line cap=round}
        ]
        \definecolor{Red}{RGB}{228, 26, 28}
        \definecolor{Blue}{RGB}{55, 126, 184}
        \definecolor{Green}{RGB}{77, 175, 74}
        \definecolor{Orange}{RGB}{255, 127, 0}
        \definecolor{Violet}{RGB}{152, 78, 163}
        \definecolor{Cyan}{RGB}{0, 200, 200}
        \definecolor{Magenta}{RGB}{220, 50, 180}
        
        \node[even_ellipse, fill=white] (C1) at (-2.5, 0) {};
        \node[even_ellipse, fill=white] (C2) at (2.5, 0) {};
        \node[even_ellipse, fill=white] (C3) at (0, -2.5) {};
        
        \begin{pgfonlayer}{background}
          \node[even_ellipse] (Special) at (0, 3) {};
        
          \foreach \cluster in {C1, C2, C3} {
            \node[dotted_circle, draw=white] (\cluster-1) at ($(\cluster) + (-.8, 0)$) {};
            \node[dotted_circle, draw=white] (\cluster-2) at (\cluster) {};
            \node[dotted_circle, draw=white] (\cluster-3) at ($(\cluster) + (.8, 0)$) {};
          }
        
          \node[dotted_circle, draw=Orange] (Special-1) at ($(Special) + (-.8, 0)$) {};
          \node[dotted_circle, draw=Violet] (Special-2) at (Special) {};
          \node[dotted_circle, draw=Cyan]   (Special-3) at ($(Special) + (.8, 0)$) {};
        
          \foreach \i in {1, 2, 3} {
            \foreach \j in {1, 2, 3} {
              \draw[base_edge, Red]   (C1-\i) to[bend right=25] (C2-\j);
              \draw[base_edge, Blue]  (C2-\i) -- (C3-\j);
              \draw[base_edge, Green] (C3-\i) -- (C1-\j);
            }
          }
        
          \foreach \targetCluster in {C1, C2, C3} {
            \foreach \m in {1, 2, 3} {
              \draw[special_edge, Orange] (Special-1) -- (\targetCluster-\m);
              \draw[special_edge, Violet] (Special-2) -- (\targetCluster-\m);
              \draw[special_edge, Cyan]   (Special-3) -- (\targetCluster-\m);
            }
          }
        \end{pgfonlayer}
        
        \node[align=center, text width=7cm] at (0, -4.5) {\small $r=3$, $q=6=\binom{r+1}{2}$, $\alpha=\tfrac14$, and\\ \small $\alpha_i = 0$ for $1\le i \le 3$, $\alpha_i=\tfrac\alpha3=\tfrac{1}{12}$ for $4\le i\le6$.};
    
        \node[short_ellipse, fill=white] (C12) at (6.5, 0) {};
        \node[short_ellipse, fill=white] (C22) at (11.5, 0) {};
        \node[short_ellipse, fill=white] (C32) at (9, -2.5) {};
        
        \begin{pgfonlayer}{background}
          \node[long_ellipse] (Special2) at (9, 3) {};
        
          \foreach \cluster in {C12, C22, C32} {
            \node[dotted_circle, draw=white] (\cluster-1) at ($(\cluster) + (-.4, 0)$) {};
            \node[dotted_circle, draw=white] (\cluster-2) at ($(\cluster) + (.4, 0)$) {};
          }
        
          \node[dotted_circle, draw=Orange]  (Special2-1) at ($(Special2) + (-1.2, 0)$) {};
          \node[dotted_circle, draw=Violet]  (Special2-2) at ($(Special2) + (-0.4, 0)$) {};
          \node[dotted_circle, draw=Cyan]    (Special2-3) at ($(Special2) + (0.4, 0)$) {};
          \node[dotted_circle, draw=Magenta] (Special2-4) at ($(Special2) + (1.2, 0)$) {};
        
          \foreach \i in {1, 2} {
            \foreach \j in {1, 2} {
              \draw[base_edge, Red]   (C12-\i) to[bend right=25] (C22-\j);
              \draw[base_edge, Blue]  (C22-\i) -- (C32-\j);
              \draw[base_edge, Green] (C32-\i) -- (C12-\j);
            }
          }
        
          \foreach \targetCluster in {C12, C22, C32} {
            \foreach \m in {1, 2} {
              \draw[special_edge, Orange]  (Special2-1) -- (\targetCluster-\m);
              \draw[special_edge, Violet]  (Special2-2) -- (\targetCluster-\m);
              \draw[special_edge, Cyan]    (Special2-3) -- (\targetCluster-\m);
              \draw[special_edge, Magenta] (Special2-4) -- (\targetCluster-\m);
            }
          }
        \end{pgfonlayer}
    
        \node[align=center, text width=7cm] at (9, -4.5) {\small $r=3$, $q=7>\binom{r+1}{2}$, $\alpha=\tfrac27$, and\\ \small $\alpha_i = 0$ for $1\le i\le 3$, $\alpha_i=\tfrac\alpha4=\tfrac{1}{14}$ for $4\le i\le7$.};
        \end{tikzpicture}
        \caption{Examples of the construction for $r=3$.}
        \label{fig:construction examples}
    \end{figure}
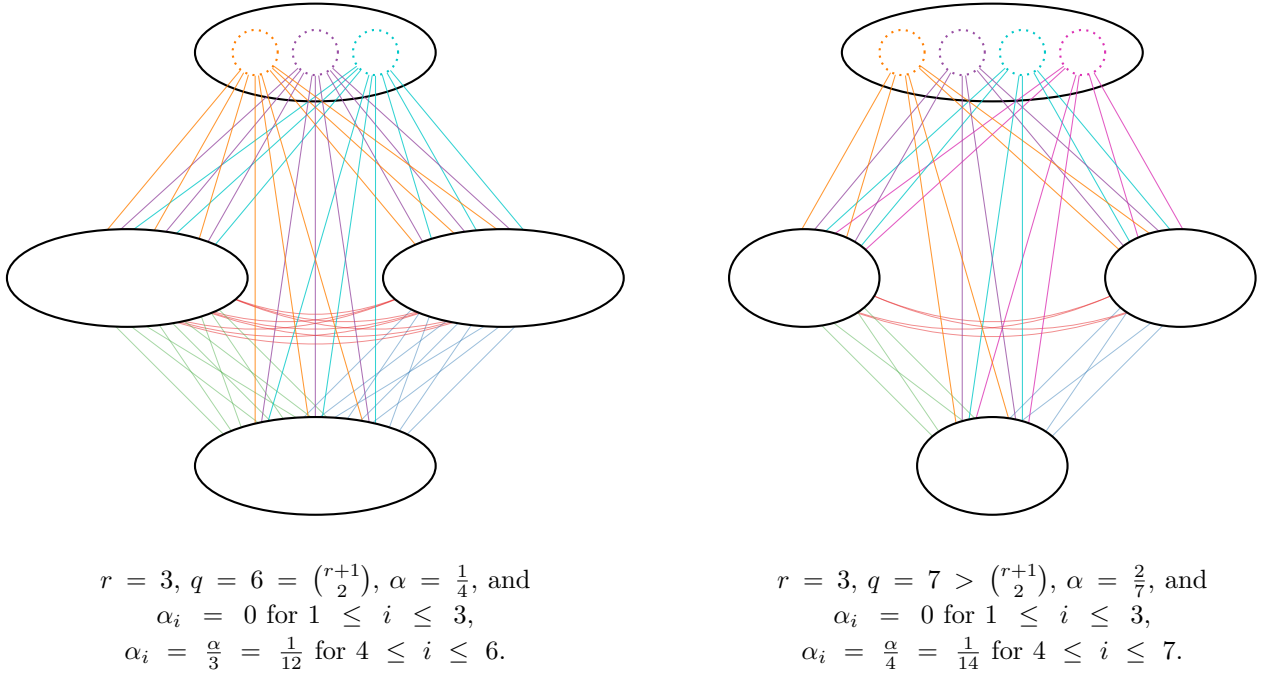
    
    We now show that in the above two constructions, there is a suitable choice of $\alpha_1,\dots,\alpha_q$ such that every $K_r$-tiling has discrepancy 0. To this end, for any $K_r$-tiling, we compute the number of edges of the $K_r$-tiling having color $i$, for each $i \in [q]$ (we will show that this number does not depend on the $K_r$-tiling).
    
    \begin{lemma}\label{lem: edge counting}
    Let $r, q, n \in \mathbb{N}$ with $r \mid n$. Let $G$ be a graph on $n$ vertices and let $\mathcal{T}$ be a $K_r$-tiling of $G$. For each $i \in [q]$, let $e_i$ denote the total number of edges of color $i$ in $\mathcal{T}$.
    \begin{enumerate}
        \item If $q \ge \binom{r}{2}$ and $G$ follows Construction \ref{construction 1}, then
        \[
            e_i =
            \begin{cases}
                \frac{n}{r}(1-2\alpha) + \alpha_i n(r-1) & \text{for } 1 \le i \le \binom{r}{2}, \\[0.5em]
                \alpha_i n(r-1) & \text{for } \binom{r}{2} < i \le q.
            \end{cases}
        \]
        \item If $q \le \binom{r}{2}$ and $G$ follows Construction \ref{construction 2}, then
        \[
            e_i =
            \begin{cases}
                (a+1)\frac{n}{r}(1-2\alpha) + \alpha_i n(r-1) & \text{for } 1 \le i \le b, \\[0.5em]
                \frac{an}{r}(1-2\alpha) + \alpha_i n(r-1) & \text{for } b < i \le q.
            \end{cases}
        \]
    \end{enumerate}
\end{lemma}
    \begin{proof}
    We start with arguments which are relevant to both constructions. 
    \begin{itemize}
        \item For $1 \leq j \leq r + 1$, let $m_j$ be the number of $K_r$'s in $\mathcal{T}$ which avoid $V_j$. Because the number of $K_r$'s in $\mathcal{T}$ which intersect $V_j$ is $|V_j|$, $m_j = \vert \mathcal{T} \vert - \vert V_j \vert = \frac{n}{r} - |V_j|$. 
        \item For $1 \leq i \leq q$, we have $e_i = e'_i + e''_i$, where $e'_i$ is the number of edges of color $i$ in $\mathcal{T}$ which does not have an endpoint in $V_{r + 1}$ and $e''_i$ is the number of those edges that does. 
        
        We compute $e''_i$ as follows: because $V_{r + 1}$ is independent, each $K \in \mathcal{T}$ uses at most one vertex from $V_{r + 1}$. The elements in the perfect tiling contributing to $e''_i$ are those $K_r$'s which contains exactly $1$ vertex from $Y_i \subset V_{r + 1}$ and such a $K_r$ contributes $(r - 1)$ edges of color $i$. Therefore, $e''_i = \vert Y_i \vert (r - 1) = \alpha_i n (r - 1)$. 
    \end{itemize}
        
    Next, we prove Part 1. Assume $G$ follows Construction \ref{construction 1}. By the above, it suffices to show that
    \[
        e'_i = 
        \begin{cases}
            \frac{n}{r}(1-2\alpha)
            & \text{when } 1 \leq i \leq \binom{r}{2},\\[1.0em] 
            0 & \text{when } \binom{r}{2} < i \leq q.    
        \end{cases}
    \]
    
    For $i > \binom{r}{2}$, no edges outside those incident to $V_{r+1}$ have color $i$, so $e'_i = 0$. Now let $1 \leq i \leq \binom{r}{2}$. There exists a unique pair $1 \leq j < k \leq r$ such that all edges in $E[V_j, V_k]$ have color $i$. Thus, $e'_i$ counts the number of cliques in $\mathcal{T}$ that contain an edge between $V_j$ and $V_k$. Because each $K \in \mathcal{T}$ contains exactly 1 vertex from all but one $V_\ell$, the number above is the sum across $\ell \neq j, k$ of the number of $K \in \mathcal{T}$ which avoids $V_i$; that is,
    \[
        e'_i = m_{r+1} + \sum_{\ell \in [r] \setminus \{j,k\}} m_\ell = \left( \frac{n}{r}-\alpha n \right) + (r-2)\frac{\alpha n}{r} = \frac{n}{r}(1-2\alpha).
    \]

    Finally, let $G$ follow Construction \ref{construction 2} and write $\binom{r}{2} = qa + b$. Then, similar to the above, for $1 \leq j < k \leq r$, the number of edges of $\mathcal{T}$ in $E[V_j,V_k]$ is $\frac{n}{r}(1-2\alpha)$. Therefore, 
    \[
        e'_i = 
        \begin{cases}
            (a+1)\frac{n}{r}(1-2\alpha)
            & \text{when } i \in \{1, \dots, b\},\\[1.0em] 
            \frac{an}{r}(1-2\alpha) & \text{when } i \in \{b+1,\dots,q\}.
        \end{cases}
    \]
    as required.
    \end{proof}
    \begin{lemma}
        Let $r, q \in \mathbb{N}$ with $r \ge 3$ and $\binom{r}{2} \leq q \leq \binom{r + 1}{2}$. For every $N \in \mathbb{N}$, there exists an integer $n > N$ divisible by $r$, a graph $G$ on $n$ vertices with minimum degree
        \[
            \delta(G) = \frac{r}{r+1}n,
        \]
        and an edge coloring $f: E(G) \to [q]$ such that every perfect $K_r$-tiling of $G$ has discrepancy 0 with respect to $f$.
        \label{lem: thm 1.6 & 1.7 lo part 1}
    \end{lemma}
    \begin{proof}
        Let $G$ follow Construction \ref{construction 1} with parameters 
        \[
        \alpha_i =
        \begin{cases}
        \frac{1}{2q}-\frac{1}{r(r+1)}, &\text{for } 1\leq i\leq \binom{r}{2},\\
        \frac{1}{2q} &\text{for } i>\binom{r}{2}.
        \end{cases}
        \]
        and $n > N$ satisfying the needed divisibility. First, we verify that these parameters are admissible. Since $q \leq \binom{r+1}{2} = \frac{r(r+1)}{2}$, we have $\frac{1}{2q} \ge \frac{1}{r(r+1)}$, which implies $\alpha_i \geq 0$ for all $i$. Summing the weights also yields:
        \[
            \alpha = \sum_{i=1}^q \alpha_i = q \cdot \frac{1}{2q} - \binom{r}{2} \frac{1}{r(r+1)} = \frac{1}{2} - \frac{r-1}{2(r+1)} = \frac{1}{r+1}.
        \]
        As $\alpha = \frac{1}{r + 1}$, by the definition of the construction we have $|V_1| = \dots = |V_r| = |V_{r + 1} | = \frac{n}{r+1}$. Thus, $G$ is a balanced complete $(r+1)$-partite graph and
        \[
            \delta(G) = n - \frac{n}{r+1} = \frac{r}{r+1}n.
        \]
        We now show that every $K_r$-tiling has discrepancy 0. Let $\mathcal{T}$ be a $K_r$-tiling, and let $e_i$ be the number of edges of colour $i$ in $\mathcal{T}$ ($1 \leq i \leq q$). Lemma \ref{lem: edge counting} implies that $e_1 = \dots = e_{\binom{r}{2}}$ and $e_{\binom{r}{2}+1}=\dots= e_{q}$, so it suffices to show that $e_1= e_q$.  
        Using the formulas in Lemma \ref{lem: edge counting} and noting $1 - \frac{2}{1 + r} = \frac{r - 1}{r + 1}$,
        \begin{align*}
            e_1 &= \frac{n}{r}(1-2\alpha) + \alpha_1 n(r-1) \\
                &= \frac{n(r-1)}{r(r+1)} + \left( \frac{1}{2q}-\frac{1}{r(r+1)} \right)n(r-1) \\
                &= n(r-1) \left[ \frac{1}{r(r+1)} + \frac{1}{2q} - \frac{1}{r(r+1)} \right] \\
                &= \frac{n(r-1)}{2q} = e_q.
        \end{align*}
        as required.
    \end{proof}
    \begin{lemma}\label{lem: thm 1.6 & 1.7 lo part 2}
        Let $r \ge 3$ and $q \ge \binom{r+1}{2}$. Then there are infinitely many integers $n$ divisible by $r$ for which there exists a graph $G$ on $n$ vertices with minimum degree
        \[
            \delta(G) = \left( \frac{1}{2} + \frac{r(r-1)}{4q} \right)n,
        \]
        and an edge coloring $f: E(G) \to [q]$ such that every perfect $K_r$-tiling of $G$ has discrepancy 0 with respect to $f$.
    \end{lemma}
    \begin{proof}
        Let $G$ follow Construction \ref{construction 1} with parameters
        \[
        \alpha_i =
        \begin{cases}
        0 &\text{for } 1\leq i\leq \binom{r}{2},\\
        \frac{1}{2q} &\text{for } \binom{r}{2} < i\leq q.
        \end{cases}
        \]
        and $n > N$ satisfying the needed divisibility. Once again, $\alpha_i \ge 0$ for $1 \leq i \leq q$ and
        \[
        \alpha = \sum_{i=1}^q \alpha_i = \frac{1}{2q}\left(q-\binom{r}{2}\right) = 
        \frac{1}{2} - \frac{r(r-1)}{4q} \leq 1.
        \]
        The condition $q \ge \binom{r+1}{2}$ implies
        \[
            \alpha = \frac{1}{2} - \frac{r(r-1)}{4q} \ge \frac{1}{2} - \frac{r - 1}{2(r + 1)} = \frac{1}{r+1}.
        \]
        As $\frac{n}{r + 1}$ is the average size of the $\vert V_i \vert$'s, this implies $|V_{r+1}| = \alpha n \ge \frac{(1-\alpha)n}{r} = |V_i|$ for all $i \in [r]$ and so
        \[
            \delta(G) = n - |V_{r+1}| = (1-\alpha)n = \left( \frac{1}{2} + \frac{r(r-1)}{4q} \right)n.
        \]
        Now, we show that every $K_r$-tiling $\mathcal{T}$ has discrepancy 0. Let $e_i$ be the number of edges of color $i$ in $\mathcal{T}$ ($1 \leq i \leq q$). By Lemma \ref{lem: edge counting} we have 
        $e_1 = \dots = e_{\binom{r}{2}}$, 
        $e_{\binom{r}{2}+1}=\dots= e_{q}$, and
        \[e_1 = \frac{n}{r} (1 - 2\alpha) = \frac{n}{r}\left(1-\left(1 - \frac{r(r-1)}{2q}\right)\right)=\frac{n(r-1)}{2q}= e_q\]
        as required.
    \end{proof}
    \begin{lemma}
        \label{lem:thm 1.8 lo}
        Let $r \ge 3$ and $2 \le q \le \binom{r}{2}$. Let $a, b \in \mathbb{N}$ such that $\binom{r}{2} = aq + b$ with $0 \le b < q$. 
        If $b=0$ or $r+b \ge q$, then there are infinitely many integers $n$ divisible by $r$ for which there exists an $n$-vertex graph $G$ with minimum degree
        \[
            \delta(G) = \frac{r}{r+1}n,
        \]
        and an edge coloring $f: E(G) \to [q]$ such that every perfect $K_r$-tiling of $G$ has discrepancy 0 with respect to $f$.
    \end{lemma}
    \begin{proof}
        Let $G$ follow Construction \ref{construction 2} with parameters (when $b = 0$, the first case is void)
        \[
        \alpha_i =
        \begin{cases}
         \frac{r+b-q}{qr(r+1)}, &\text{for } 1\leq i\leq b\\
        \frac{r+b}{qr(r+1)} &\text{for } b<i\leq q
        \end{cases}
        \]
        and $n > N$ satisfying the divisibility conditions. Clearly, each $\alpha_i \ge 0$ and
        \[
        \alpha = 
        \sum_{i=1}^q \alpha_i = 
        b\frac{r+b-q}{qr(r+1)}+(q-b)\frac{r+b}{qr(r+1)}=\frac{rb+b^2-qb+qr-br+qb-b^2}{qr(r+1)} =\frac{1}{r+1}.
        \]
        So $G$ is a balanced complete $(r + 1)$-partite graph and $\delta(G) = \frac{r}{r + 1}n$.

        We now demonstrate that every perfect $K_r$-tiling $\mathcal{T}$ has discrepancy 0. This is immediate when $b = 0$ by Lemma \ref{lem: edge counting} so assume otherwise. Let $e_i$ denote the number of edges of color $i$ ($1 \leq i \leq q$). By Lemma \ref{lem: edge counting}, we have $e_1 = \dots = e_b$ and $e_{b+1} = \dots = e_q$ so it suffices to prove that $e_1 = e_q$. Indeed:
        \begin{align*}
            e_1 &= (a+1)\frac{n}{r}\left(1-\frac{2}{r+1}\right)+n(r-1)\frac{r+b-q}{qr(r+1)}\\
            &=(a+1)\frac{n}{r} \cdot \frac{r-1}{r+1} + 
            n(r-1)\frac{r+b-q}{qr(r+1)}\\
            &=a\frac{n}{r} \cdot \frac{r-1}{r+1} + 
            n(r-1)\frac{r+b}{qr(r+1)}
            = e_q.
        \end{align*}
    \end{proof}

\section{The Key Lemmas}\label{sec:lemmas}

    In this section we state and prove the key lemmas used in the proofs of our results (except for Lemma \ref{lem:main}, whose proof is deferred to Section \ref{sec: proof of main lemma}). We start with some notation.
    A {\em $q$-edge-colored graph} is a pair $(G,f)$ where $G$ is a graph and $f : E(G) \rightarrow [q]$ is a $q$-coloring of $E(G)$. 
    For a subgraph $F \subseteq G$, we denote by $\#c(F)$ the number of edges of $F$ having color $c$. The {\em color profile} of $F$ is the vector $(\#c(F) : c \in [q])$.


    \subsection{Templates}
    
    Key objects in our proofs are graphs that have two different $K_r$-tilings, which have different color profiles. Such gadgets have been used in multiple previous works on discrepancy. Following \cite{BCPT:21}, we call such graphs {\em templates}. We now define the specific templates that we will use. 
    \begin{definition}[$K_r$-template]\label{def: template}
        An even cycle $C = w_1 \dots w_{2k} w_1$ is \textbf{balanced} if 
        \[\{f(w_1w_2), f(w_3w_4), \dots, f(w_{2k-1}w_{2k})\} = \{f(w_2w_3), f(w_4w_5), \dots, f(w_{2k}w_1)\}\]
        as multisets and \textbf{unbalanced} otherwise.
        
        For $r \ge 3$, a $K_r$\textbf{-template} is a subgraph consisting of a copy $K$ of $K_{r-2}$ (the \emph{center}) together with an unbalanced cycle of length 4 or 6 in $N(K)$. \footnote{We could also use longer cycles, i.e., the resulting graphs would still have the key property given by Lemma \ref{lem:template}. But it turns out that forbidding templates with cycles of lengths 4 and 6 suffices to infer enough structural information on the host graph. Hence, we do not use longer cycles.}
    \end{definition}

    The importance of templates is in the fact that by blowing up a template in an appropriate way, one can find two $K_r$-tilings of the blowup which have different color profiles. Later on, we will use this to find a $K_r$-tiling with large discrepancy in the case where there exist many vertex-disjoint copies of this blowup. The definition of the blowup is as follows. 
    \begin{definition}[$F^+$]\label{def:template blowup}
    Let $F$ be a $K_r$-template with cycle length $2k \in \{4,6\}$. Let $F^+$ be the edge-colored graph obtained from $F$ by blowing up each vertex in the center of $F$ to size $k$. This means that each vertex in the center of $F$ is replaced with an independent set of size $k$ (and the cycle vertices, i.e., the vertices outside the center, remain unchanged). Also, every edge is replaced with a complete bipartite graph of the same color.
    \end{definition}

    The following lemma states that for a $K_r$-template $F$, its blowup $F^+$ indeed has two $K_r$-tilings having different color profiles.

    \begin{lemma}\label{lem:template}
    Let $F$ be a $K_r$-template, and let $F^+$ be as in Definition \ref{def:template blowup}. Then there are two $K_r$-tilings $\mathcal{T}^1,\mathcal{T}^2$ of $F^+$ and there is a color $c \in [q]$, such that $c$ appears a different number of times in $\mathcal{T}^1$ and $\mathcal{T}^2$.
    \end{lemma}
    \begin{proof}
        Let $2k \in \{4,6\}$ be the length of the cycle in $F$. Let $v_1, \dots, v_{r-2}$ be the vertices of the center of $F$, and let $u_1, \dots, u_{2k}$ be the vertices of the unbalanced cycle, in order.
        
        In $F^+$, each center vertex $v_i$ is replaced by an independent set $V_i = \{v_i^{(1)}, \dots, v_i^{(k)}\}$ of size $k$. For each $j \in [k]$, let $Z_j = \{v_1^{(j)}, \dots, v_{r-2}^{(j)}\}$ be the $j$-th copy of the center. We define $\mathcal{T}^1$ and $\mathcal{T}^2$ as follows (cycle indices taken modulo $2k$):
        \[
            \mathcal{T}^1 = \left\{ Z_j \cup \{u_{2j-1}, u_{2j}\} \;\middle|\; 1 \le j \le k \right\}, \; \mathcal{T}^2 = \left\{ Z_j \cup \{u_{2j}, u_{2j+1}\} \;\middle|\; 1 \le j \le k \right\}.
        \]
        These are two $K_r$-tilings of $F^+$.
        The edges in these tilings can be partitioned into three types:
        \begin{enumerate}
            \item \emph{Edges within the center copies:}
            We first consider the edges contained in 
            $Z := Z_1 \cup \dots \cup Z_k$.
            Both $\mathcal{T}^1$ and $\mathcal{T}^2$ contain the edges of the clique $Z_j$ for each $j \in [k]$, 
            and each edge of $\mathcal{T}^1$ or $\mathcal{T}^2$ which is contained in $Z$ belongs to $Z_j$ for some $j \in [k]$.
            Thus, $\mathcal{T}^1$ and $\mathcal{T}^2$ contain the same edges within $Z$.
            \item \emph{Edges between center copies and the cycle:} For each $j \in [k]$, the edges connecting $Z_j$ to $u_{2j}$ appear in both $\mathcal{T}^1$ and $\mathcal{T}^2$, so their colors are identical. The remaining edges connect $Z_j$ to $u_{2j-1}$ in $\mathcal{T}^1$ and to $u_{2j+1}$ in $\mathcal{T}^2$. As $j$ ranges from $1$ to $k$, both $\{2j-1\}$ and $\{2j+1\}$ traverse the same set of indices modulo $2k$ (namely, all odd indices). 
            Also, for all $1 \leq i,i' \leq k$ and $1 \leq \ell \leq 2k$, the sets $E(Z_i,u_{\ell})$ and $E(Z_{i'},u_{\ell})$ have the same color profile.
            Thus, the edges between the center copies and the cycle have the same color profiles in $\mathcal{T}^1$ and $\mathcal{T}^2$.
            \item \emph{Edges within the cycle:} $\mathcal{T}^1$ contains the edges $\{u_1u_2, u_3u_4, \dots, u_{2k-1}u_{2k}\}$, while $\mathcal{T}^2$ contains the edges $\{u_2u_3, u_4u_5, \dots, u_{2k}u_1\}$. Since $F$ is a $K_r$-template, the cycle $u_1,\dots,u_{2k}$ is unbalanced. Therefore, the multisets of colors for these two edge-sets differ.
        \end{enumerate}
        Since the edge color multisets agree for types 1 and 2, but differ for type 3, the edge color multisets for $\mathcal{T}^1$ and $\mathcal{T}^2$ are distinct.
    \end{proof}

    \subsection{The Cleaning Lemma}
	In the proof of upper bounds for Theorems \ref{thm:general upper bound}, \ref{thm:r=4}, and \ref{thm:r=3}, we distinguish between two cases, depending on whether $G$ has many or few templates. If $G$ has many templates, then we will use this directly to obtain a $K_r$-tiling with high discrepancy. Otherwise, we will show that one can delete a small number of vertices and edges to obtain a graph with no templates and still having large minimum degree, matching the assumptions of the Main Lemma (Lemma \ref{lem:main}). These steps are given by the following lemma, which we call the {\em cleaning lemma}.
    \begin{lemma}[Cleaning lemma]\label{lem:template cleaning}
        Let $r \geq 3$. 
        For every $\xi > 0$ there is $\zeta > 0$ such that the following holds. Let $(G, f)$ be an edge-colored graph with $n \gg 1/\xi$ vertices, where $n$ divisible by $r$, and with $\delta(G) \geq (\frac{r-1}{r} + \xi)n$. Then one of the following holds:
        \begin{enumerate}
            \item There is a $K_r$-tiling $\mathcal{T}$ of $G$ and a color $c \in [q]$ such that $c$ appears on at least a $(\frac{1}{q}+\zeta)$-fraction of the edges of $\mathcal{T}$.
            \item There is a subgraph $G'$ of $G$ such that 
            $|V(G')| \geq (1-\xi)n$, 
            $\delta(G') \geq \delta(G) - \xi n$, 
            $G'$ has no $K_r$-templates, and the set $V(G) \setminus V(G')$ has a $K_r$-tiling (in $G$).
        \end{enumerate}
    \end{lemma}
    \noindent The proof of Lemma \ref{lem:template cleaning} requires several tools. First, we need the following Absorbing Lemma from \nolinebreak\cite{Treglown16}. 
 
    \begin{lemma}[Absorbing Lemma~\cite{Treglown16}]\label{lemma:absorb}
    Let $0 < 1/n \ll \nu \ll \xi \ll 1/r$ where $n, r \in \mathbb{N}$ and $r \geq 2$. Let $G$ be a graph on $n$ vertices with $\delta(G) \geq (\frac{r-1}{r} + \xi)n$. Then there exists $M \subseteq V(G)$ with 
    $|M| \leq \nu n$ such that for every $W \subseteq V(G) \setminus M$ with $|W| \leq \nu^3 n$ and with $|W|$ divisible by $r$, the graph $G[M \cup W]$ has a $K_r$-tiling.
    \end{lemma}

    We will also need a multicolor version of the Graph Removal Lemma \cite{RS:78} (see also the survey \cite{ConlonFox_survey}). 
    Here we consider $q$-edge-colored graphs.
    It is easy to see that the standard proof of the removal lemma generalizes to this multicolor setting, giving the following:
    \begin{lemma}[Multicolor removal lemma]\label{lem:removal lemma}
    Let $\mathcal{F}$ be a finite family of $q$-edge-colored graphs. For every $\varepsilon > 0$ there exists $\delta > 0$ such that the following holds. Let $G$ be an $n$-vertex graph with $n \gg 1/\varepsilon$, such that for every $F \in \mathcal{F}$,
    $G$ contains at most $\delta n^{v(F)}$ copies of $F$.
    Then one can delete at most $\varepsilon n^2$ edges of $G$ to obtain a graph with no copy of any $F \in \mathcal{F}$.
    \end{lemma}

    Finally, we need the following well-known fact, which follows from the hypergraph version of the K\H{o}v\'ari-S\'os-Tur\'an theorem, due to Erd\H{o}s~\cite{erdos:64b}. 
    \begin{lemma}\label{lem:supersaturation}
        Let $F$ be an edge-colored graph, and let $F'$ be obtained by blowing up some of the vertices of $F$.\footnote{This means that we replace each vertex of $F$ with an independent set (possibly of size 1), and replace each edge $e$ of $F$ with a complete bipartite graph of the same color as $e$.}
        For every $\delta > 0$ there is $\eta > 0$ such that if $G$ is a graph on $n \gg \frac{1}{\delta}, v(F')$ vertices with at least $\delta n^{v(F)}$ copies of $F$, then $G$ contains at least $\eta n^{v(F')}$ copies of $F'$. 
    \end{lemma}
    \begin{proof}[Proof sketch]
        Write $V(F) = \{v_1,\dots,v_k\}$, and let $a_i$ be the size of the blowup-set replacing $v_i$ in $F'$.
        Consider a $k$-uniform hypergraph on $V(G)$ whose hyperedges correspond to copies of $F$. Then by assumption, $e(G) \geq \delta n^{k}$. By \cite{erdos:64b} (or rather, the supersaturation version thereof), this hypergraph contains at least $\eta n^{v(F')}$ copies of the complete $k$-partite $k$-graph $K^{(k)}_{a_1,\dots,a_k}$, and $\eta > 0$ depends only on $\delta$. Each such copy corresponds to a copy of $F'$ in $G$.
    \end{proof}

    We are now ready to prove Lemma \ref{lem:template cleaning}.
    We note that unlike \cite{BCPT:21}, we avoid (explicitly) using the Szemer\'edi regularity lemma, instead using the removal lemma (whose proof relies on the regularity lemma). We believe that this makes the presentation cleaner. 

    \begin{proof}[Proof of Lemma \ref{lem:template cleaning}]
    Let $\mathcal{F}$ be the family of $K_r$-templates. Note that every $K_r$-template has either $(r-2) + 4 = r+2$ or $(r-2) + 6 = r+4$ vertices. In particular, $\mathcal{F}$ is finite. Fix constants 
    $$
    1/n \ll \zeta \ll \eta \ll \delta \ll \nu \ll \xi \ll 1/r.
    $$
    We consider two cases.

    \paragraph{Case 1:}
    There is $F \in \mathcal{F}$ such that $G$ has at least $\delta n^{v(F)}$ copies of $F$. 
    Let $F^+$ be defined as in Definition \ref{def:template blowup}. For convenience, put $t := |V(F^+)|$.
    By Lemma \ref{lem:supersaturation},
    $G$ contains at least $\eta n^{t}$ copies of $F^+$. Note that every vertex can participate in at most $tn^{t-1}$ copies of $F^+$.
    Now, let $F_1,\dots,F_m$ be a maximal collection of vertex-disjoint copies of $F^+$ in $G$. We claim that $m \geq \frac{1}{t^2} \eta n$. Indeed, the vertices in $\bigcup_{i=1}^m V(F_i)$ intersect every copy of $F^+$ (by the maximality of $m$), and hence the total number of copies of $F^+$ is at most $m \cdot t \cdot tn^{t-1}$. As this is at least $\eta n^t$, it follows that $m \geq \frac{1}{t^2} \eta n$. From now on, we will consider a collection $F_1,\dots,F_m$ with $m = \lceil \frac{1}{t^2}\eta n \rceil$.

    Let $G'$ be the graph obtained from $G$ by deleting $\bigcup_{i=1}^m V(F_i)$. Then $|V(G) \setminus V(G')| \leq mt \leq \eta n \leq \frac{\xi}{2}n$, and so $\delta(G') \geq \delta(G) - \frac{\xi}{2}n > \frac{r-1}{r}|V(G')|$. By the Hajnal-Szemer\'edi theorem~\cite{HS:70}, $G'$ has a $K_r$-tiling $\mathcal{T}_0$. By Lemma \ref{lem:template}, for every $1 \leq i \leq m$, there is a color $c_i \in [q]$ and two $K_r$-tilings $\mathcal{T}_i^1,\mathcal{T}_i^2$ of $F_i$, such that $c_i$ appears a different number of times in $\mathcal{T}_i^1$ and $\mathcal{T}_i^2$. By averaging, there is $c \in [q]$ such that $c_i = c$ for at least $m/q$ of the indices $i \in [m]$. Without loss of generality, we may assume that $c_i = c$ for $i = 1,\dots,\lceil m/q \rceil$, and that for each such $i$, $\mathcal{T}_i^1$ has more edges of color $c$ than 
    $\mathcal{T}_i^2$. Let $\mathcal{T}$ be the $K_r$-tiling of $G$ consisting of $\mathcal{T}_0$ and $\mathcal{T}_i^1$ for every $1 \leq i \leq m$. Let $\mathcal{T}'$ be the $K_r$-tiling obtained from $\mathcal{T}$ by replacing $\mathcal{T}_i^1$ with $\mathcal{T}_i^2$ for every $1 \leq i \leq \lceil m/q \rceil$; in other words, $\mathcal{T}'$ consists of $\mathcal{T}_0$ and of $\mathcal{T}_i^2$ for every $1 \leq i \leq \lceil m/q \rceil$. 
    Then $\#c(\mathcal{T}) - \#c(\mathcal{T}') \geq m/q \geq \frac{1}{qt^2}\eta n \geq qr\zeta n$. 
    Hence, one of $\mathcal{T},\mathcal{T}'$ satisfies the assertion of Item 1 of the lemma. Indeed, 
    if $\#c(\mathcal{T}) \geq (\frac{1}{q}+\zeta) \cdot e(\mathcal{T})$ then we are done, and otherwise, $\#c(\mathcal{T}') \leq 
    (\frac{1}{q}+\zeta) \cdot e(\mathcal{T}') - qr\zeta n \leq 
    (\frac{1}{q} - (q-1)\zeta) \cdot e(\mathcal{T}')$. 
    Here we used that $e(\mathcal{T}') = \frac{(r-1)n}{2} \leq rn$.
    By averaging, there is a color $c' \in [q] \setminus \{c\}$ such that $\#c'(\mathcal{T}') \geq (\frac{1}{q}+\zeta) \cdot e(\mathcal{T}')$, as required.

    \paragraph{Case 2:}  
    For every $F \in \mathcal{F}$, $G$ has at most $\delta n^{v(F)}$ copies of $F$. 
    By Lemma \ref{lem:removal lemma} with parameter $\frac{\nu^3 \xi}{8}$ in place of $\varepsilon$, we can delete a set $E$ of at most 
    $\frac{\nu^3\xi}{8}n^2$ edges of $G$ and thus obtain a subgraph with no $K_r$-templates. 
    Let $M \subseteq V(G)$ be given by Lemma \ref{lemma:absorb}. 
    Let $W$ be the set of vertices $v \in V(G) \setminus M$ such that $v$ touches at least $\frac{\xi}{2}n$ edges of $E$. As 
    $|E| \leq \frac{\nu^3\xi}{8}n^2$, we have $|W| \leq \frac{\nu^3}{2}n$. Add less than $r$ vertices to $W$ to make sure that $|W|$ is divisible by $r$. 
    Now $|W| \leq \nu^3 n$.
    Let $G'$ be the subgraph of $G$ obtained by deleting the edges in $E$ and the vertices in $M \cup W$. Clearly $G'$ has no $K_r$-templates. Also, $|V(G) \setminus V(G')| = |M| + |W| \leq \nu n + \nu^3 n \leq \frac{\xi}{2}n$. By the definition of $W$, every $v \in V(G')$ has $d_{G'}(v) \geq d_G(v) - \frac{\xi}{2}n - 
    |V(G) \setminus V(G')| \geq d_G(v) - \xi n$. This shows that $\delta(G') \geq \delta(G) - \xi n$. Finally, by the guarantees of Lemma \ref{lemma:absorb}, $V(G) \setminus V(G') = M \cup W$ has a $K_r$-tiling (in $G$). Hence, Item 2 in Lemma \ref{lem:template cleaning} holds. 
    \end{proof}


    \subsection{Graphs with No Templates}
    Lemma \ref{lem:template cleaning} will allow us to assume that the (edge-colored) host graph $(G,f)$ has no $K_r$-templates. Thus, the next step is to study the structure of edge-colored graphs having large minimum degree and no templates. We will prove the following two lemmas. Lemma \ref{lem:main} states that if $(G,f)$ has no $K_r$-templates and $\delta(G) > \frac{r-1}{r}n + C$ then there is a large set $U \subseteq V(G)$ which contains only few colors. Lemma \ref{lem:aux_main} shows that under a stronger assumption on the minimum degree, namely that $\delta(G) > \frac{r}{r+1}n + C$, the set $U$ in fact contains only one color. 

	\begin{lemma}\label{lem:main}
		Let $r \geq 3$ and let $(G, f)$ be an edge-colored graph with $\delta(G) \geq \frac{r-1}{r}|G| + 3$ with no $K_r$-templates. Then there is $U \subseteq V(G)$ with $|U| \geq \delta(G)$ such that the edges of $G[U]$ have at most $\binom{r}{2}$ different colors. 
	\end{lemma}

    \begin{lemma}\label{lem:aux_main}
		Let $r \geq 3$ and let $(G, f)$ be an edge-colored graph with $\delta(G) \geq \frac{r}{r + 1}|G| + 3$ with no $K_r$-templates. Then there exists a subset $U \subseteq V(G)$ with $|U| \geq \delta(G)$ such that $G[U]$ is monochromatic.
	\end{lemma}

    Lemma \ref{lem:main} will be used in the proofs of Theorems \ref{thm:r=4} and \ref{thm:r=3}, and Lemma \ref{lem:aux_main} will be used in the proof of Theorem \ref{thm:general upper bound}.
    The proofs of Lemmas \ref{lem:main} and \ref{lem:aux_main} are by far the most involved part of our argument. We defer these proof to Section \ref{sec: proof of main lemma}.

    \subsection{Extremal Properties of Tilings}
    We end Section \ref{sec:lemmas} with one more lemma which is used in the proofs of our main results. 
    This is Lemma \ref{lem: extremal tiling main} below, which shows that if $\mathcal{T}$ is a $K_r$-tiling of $G$ and $U \subset V(G)$ is a large subset, then a significant proportion of the edges of $\mathcal{T}$ must be contained in $U$. This will then be combined with Lemmas \ref{lem:main} and \ref{lem:aux_main} to show that under the conclusion of these lemmas, every $K_r$-tiling has large discrepancy (for a suitable regime of the parameter $q$). We start with the following auxiliary lemma.
    
    \begin{lemma}\label{lem: extremal tiling fact}
        Let $r, n \in \mathbb{N}$ with $r \mid n$, and let $V(K_n) = X \sqcup Y$ be a vertex partition. Let $\mathcal{T}$ be a $K_r$-tiling of $K_n$ that minimizes the number of edges in $\mathcal{T}$ induced by $X$ (i.e., edges with both endpoints in $X$). Then there exists an integer $k \ge 0$ such that for every $r$-clique $K \in \mathcal{T}$, $|V(K) \cap X| \in \{k, k+1\}$.
    \end{lemma}
    \begin{proof}
        For $K \in \mathcal{T}$, let 
        $\ell(K) := |K \cap X|$. Suppose the lemma is false. Then there exist $K, K' \in \mathcal{T}$ such that $\ell(K') \ge \ell(K) + 2$. Put $i = \ell(K)$ and $j = \ell(K')$.
    
        Since $i < j$, there must exist $y \in K \cap Y$ and $x \in K' \cap X$. We construct a new tiling $\mathcal{T}'$ by swapping these vertices. Namely, define two new vertex sets:
        \[ K_{new} := (K \setminus \{y\}) \cup \{x\}, \quad K'_{new} := (K' \setminus \{x\}) \cup \{y\}. \]
        So $K_{new}$ and $K'_{new}$ are $r$-cliques (in $K_n$). Now, let $\mathcal{T}' := (\mathcal{T} \setminus \{K, K'\}) \cup \{K_{new}, K'_{new}\}$.
    
        The new intersection sizes are $\ell(K_{new}) = i+1$ and $\ell(K'_{new}) = j-1$. Consequently, the change in the number of edges induced by $X$ is
        \[ \left[ \binom{i+1}{2} + \binom{j-1}{2} \right] - \left[ \binom{i}{2} + \binom{j}{2} \right] = i - j + 1. \]
        Since $j \ge i + 2$, we have $i - j + 1 < 0$, which contradicts the minimality of $\mathcal{T}$.
    \end{proof}
    \begin{lemma}\label{lem: extremal tiling main}
    Let $r, n \in \mathbb{N}$ with $r \mid n$, and let $V(K_n) = X \sqcup Y$ be a vertex partition such that $|X| > \frac{r-1}{r} n$. Then for every $K_r$-tiling $\mathcal{T}$ of $K_n$, the fraction of edges of $\mathcal{T}$ induced by $X$ is at least $\frac{2|X|}{n} - 1$.
    \end{lemma}
    \begin{proof}
        It suffices to consider a $K_r$-tiling $\mathcal{T}$ which minimizes the number of edges induced by $X$. So let $\mathcal{T}$ be such a $K_r$-tiling.
        By Lemma \ref{lem: extremal tiling fact}, there exists $k \in \{0, \dots, r-1\}$ such that 
        $|K \cap X| \in \{k, k+1\}$ for every $K \in \mathcal{T}$.
        Since $|\mathcal{T}| = n/r$, we have 
        $|X| \le \frac{n}{r}(k+1)$. The condition 
        $|X| > \frac{r-1}{r} n$ now implies $k + 1 > r - 1$. As $k$ is an integer strictly less than $r$, we must have $k = r-1$. Thus, every clique in $\mathcal{T}$ intersects $Y$ in either $1$ vertex or $0$ vertices.
        
        Since the cliques in $\mathcal{T}$ are vertex-disjoint and cover $Y$, exactly $|Y|$ cliques intersect $Y$ (each containing exactly one vertex from $Y$). Each of these cliques contributes $\binom{r-1}{2}$ edges to $X$. The remaining $\frac{n}{r} - |Y|$ cliques are contained entirely in $X$ and contribute $\binom{r}{2}$ edges each.
        Letting $e_X(\mathcal{T})$ denote the number of edges in $\mathcal{T}$ induced by $X$, we have
        \[ e_X(\mathcal{T}) = |Y| \binom{r-1}{2} + \left(\frac{n}{r} - |Y|\right) \binom{r}{2} = \frac{n}{r}\binom{r}{2} - (r-1)|Y|. \]
        The total number of edges in $\mathcal{T}$ is $e(\mathcal{T}) = \frac{n}{r}\binom{r}{2} = \frac{n(r-1)}{2}$. Dividing $e_X(\mathcal{T})$ by $e(\mathcal{T})$ completes the \nolinebreak proof:
        \[ \frac{e_X(\mathcal{T})}{e(\mathcal{T})} = 1 - \frac{(r-1)|Y|}{n(r-1)/2} = 1 - \frac{2(n - |X|)}{n} = \frac{2|X|}{n} - 1. \qedhere \]
    \end{proof}
	
	\section{Proof of Lemmas \ref{lem:main} and \ref{lem:aux_main}}
    \label{sec: proof of main lemma}
    As mentioned above, the proofs of Lemmas \ref{lem:main} and \ref{lem:aux_main} are by far the most involved part of our argument. The reader may therefore first skip to Section \ref{sec:proof of theorems}, where these lemmas are applied (together with the other lemmas from Section \ref{sec:lemmas}) to prove the main results of the paper.   
    
    Let us give a brief overview of the proof. Recall that our goal is to find a large set $U$ which contains only few colors (only $\binom{r}{2}$ colors for Lemma \ref{lem:main} and only 1 color for Lemma \ref{lem:aux_main}). We have two strategies for finding $U$: 
    \begin{enumerate}
        \item Pick $U$ to be the neighborhood of some vertex $v$.
        \item Pick $U$ to consist of all vertices having $r$ neighbors in a fixed $(r+1)$-clique $K$.
    \end{enumerate}
    These two cases are handled in Sections \ref{subsec:U = N(v)} and \ref{subsec:U derived from K}, respectively. In Case 2, we show that under certain conditions, the colors contained in $U$ are precisely those contained in $K$. Also, whenever applying Case 2, we assume that $K$ contains at most $\binom{r}{2}$ colors, giving the required result. 

    Case 1 is more complicated. To show that $U = N(v)$ contains only few colors, we proceed as follows. First, we fix an $r$-clique $K \subseteq N(v)$ with certain properties. Then, given any edge $uw$ inside $N(v)$, we find a {\em chain} of $r$-cliques $K_1,\dots,K_m \subseteq N(v)$ such that $K_1 = K$ and $u,w \in K_m$, and $|K_i \cap K_{i+1}| = r-2$ for all $i$. We then argue that for every $i$, $K_i$ and $K_{i+1}$ have the same color profile. Namely, we can ``transfer" the colors from $K_i$ to $K_{i+1}$. 
    Applying this for all $i=1,\dots,m-1$, we conclude
    that every color appearing in $K_m$ also appears in $K_1 = K$. But $K_m$ contains $u,w$, so the color of $uw$ is one of the at most $\binom{r}{2}$ appearing in $K$ (as $|K|=r$). Since $uw$ was an arbitrary edge in $N(v)$, this shows that $N(v)$ contains only $\binom{r}{2}$ colors. (The argument for Lemma \ref{lem:aux_main} is similar). We note that for Lemma \ref{lem:main}, we have several different settings where the above argument can be carried out, which are applied in different cases; see Lemmas \ref{lem:very simple K_{r+1}}, \ref{lem:K_{r+1} with many colors} and \ref{lem:K_{r+1} with monochromatic vertex}. The analysis of the aforementioned color-transferal is given in Section \ref{subsec:transfering colors}. This analysis uses as a subroutine some observations on the possible colorings of a {\em bowtie} contained in the common neighborhood of an $(r-2)$-clique (a bowtie consists of two triangles sharing a vertex). Bowties are handled in Section \ref{subsec: bowties}. Finally, we prove the existence of the necessary clique-chains (as mentioned above) in Section \nolinebreak \ref{subsec:chains}.

    We will need the following simple lemma.
    \begin{lemma}[Minimum degree in neighborhood]\label{lem:min degree in neighbourhood}
		Let $r \geq 2$ and let $G$ be a graph satisfying $\delta(G) \geq \frac{r-1}{r}|G| + C$ for some $C \geq 0$. Then for every $v \in V(G)$, it holds that $\delta(G[N(v)]) \geq \frac{r-2}{r-1}|N(v)| + C$. 
	\end{lemma}
	\begin{proof}
		  Let $v \in V(G)$ and let $u \in N(v)$. We have
          \[\begin{aligned}
              d_{N(v)}(u) & = |N(u) \cap N(v)| = d(u) - |N(u) \setminus N(v)| \ge d(u) - (\vert G \vert - d(v))\\
              & \ge \frac{r - 1 }{r} \vert G \vert + C -\vert G \vert + d(v) 
              = d(v) - \frac{1}{r} |G| + C\\
              & \ge \frac{r-2}{r-1} d(v) + C,
          \end{aligned}\]
          where in the last inequality we used that
          $d(v) \geq \frac{r-1}{r}|G|$. This proves the lemma.
	\end{proof}
    \noindent
    By iterating the above lemma, we get the following:
    \begin{lemma}[Minimum degree in neighborhood of $K_t$] \label{lem: min degree in neighbourhood of K}
        Let $r \ge 2$ and let $G$ be a graph satisfying $\delta(G) \ge \frac{r - 1}{r} \vert G \vert + C$ for some $C \geq 0$. Let $1 \leq t \leq r - 1$. Then for every copy $K$ of $K_t$ in $V(G)$, it holds that $\delta(G[N(K)]) \ge \frac{r - t - 1}{r - t}\vert N(K) \vert + C$.
    \end{lemma}
    \begin{proof}
        By induction on $t$.
        When $t = 1$, this is just Lemma \ref{lem:min degree in neighbourhood}. For the induction step, suppose the statement holds for some $t \leq r - 2$, and let $K$ be a copy of $K_{t + 1}$ with vertices $v_1, v_2, \dots, v_{t + 1}$. Put $K' = K - v_{t + 1}$ and $G' = G[N(K')]$, so that $K'$ is a copy of $K_t$ in $G$ and $v_{t + 1} \in G'$. By the induction hypothesis, $\delta(G') \ge \frac{r - t - 1}{r - t}\vert G' \vert + C$. 
        Note that $G[N(K)]$ is the subgraph of $G'$ induced by $N_{G'}(v_{t+1})$.
        Also, we have $r - t \ge 2$ (as $t \leq r - 2$). Lemma \ref{lem:min degree in neighbourhood} (applied to $G'$ and $v_{t+1}$) gives 
        \[\delta(G[N(K)]) \ge \frac{r - t - 2}{r - t - 1} \cdot \vert N(K) \vert + C, 
        \]
        establishing the induction step.
    \end{proof}
	\subsection{Bowties}\label{subsec: bowties}
    \begin{definition}
        Let $G$ be a graph. A \textbf{bowtie} is a tuple $(v, x_1y_1, x_2y_2)$ such that $v,x_1,y_1,x_2,y_2$ are distinct vertices of $G$ and $\{v, x_1, y_1\}$ and $\{v, x_2, y_2\}$ are triangles in $G$.
    \end{definition}
    We now show that if an edge-colored graph $(G,f)$ has minimum degree above $\frac{r-1}{r}|G|$ and no $K_r$-templates, then every bowtie in $G$ which is contained in the common neighborhood of a $K_{r-2}$ has a constraint on its colors. This is established in the following lemma.
    \begin{lemma}[Bowtie lemma]\label{lem: Bowtie fact} Let $r \ge 3$ and let $(G, f)$ be an edge-colored graph satisfying $\delta(G) \ge \frac{r-1}{r}|G| + 3$ and having no $K_r$-template. Let $K$ be a copy of $K_{r-2}$ in $G$, and let $(v, x_1y_1, x_2y_2)$ be a bowtie contained in $N(K)$. Then
    \begin{equation}\label{eq:bowtie lemma}
        \{f(x_1y_1), f(vx_2), f(vy_2)\} = \{f(x_2y_2), f(vx_1), f(vy_1)\}
    \end{equation}
    as multisets.
    \end{lemma}
    \begin{proof}
    Let $G' := N[K]$. Inserting $t = r - 2$, $C = 3$ in Lemma \ref{lem: min degree in neighbourhood of K}, we get $\delta(G') \ge \frac{1}{2} \cdot \vert G' \vert + 3$. This implies that any two vertices in $G'$ have at least $6$ common neighbors (in $G'$). Hence, there exist distinct 
    $w_x, w_y \in V(G')$ such that $w_x, w_y \notin \{v, x_1, x_2, y_1, y_2\}$, $w_x$ is a common neighbor of $x_1,x_2$, and $w_y$ is a common neighbor of $y_1,y_2$.
    See Figure \ref{fig:bowtie_lemma_proof} for an illustration.
    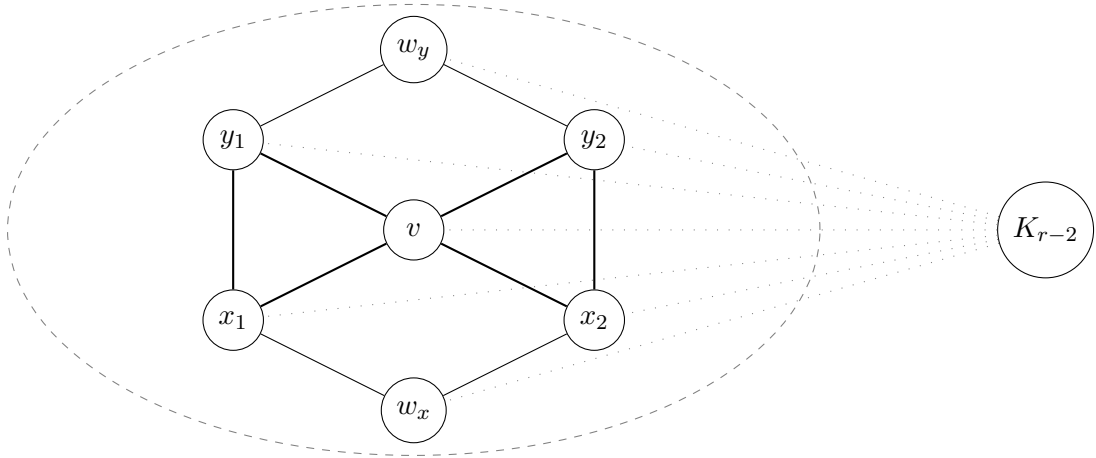
\begin{figure}[H]
        \centering
        \begin{tikzpicture}[scale=1.2, every node/.style={circle, draw, minimum size=0.8cm}]

            \node (y1) at (-2,2) {$y_1$};   
            \node (x1) at (-2,0) {$x_1$};
            \node (v) at (0,1) {$v$};
            \node (y2) at (2,2) {$y_2$};   
            \node (x2) at (2,0) {$x_2$};
            \node (wy) at (0,3) {$w_y$};
            \node (wx) at (0,-1) {$w_x$};
            \node (K) at (7,1) {$K_{r - 2}$};

            \draw[thick, black] (v) -- (y2);
            \draw[thick, black] (v) -- (x2);
            \draw[thick, black] (v) -- (x1);
            \draw[thick, black] (v) -- (y1);
            \draw[thick, black] (y1) -- (x1);
            \draw[thick, black] (y2) -- (x2);

            \draw (wy) -- (y1);
            \draw (wy) -- (y2);
            \draw (wx) -- (x1);
            \draw (wx) -- (x2);

            \draw[loosely dotted, gray] (K) -- (x2);
            \draw[loosely dotted, gray] (K) -- (x1);
            \draw[loosely dotted, gray] (K) -- (v);
            \draw[loosely dotted, gray] (K) -- (y1);
            \draw[loosely dotted, gray] (K) -- (y2);
            \draw[loosely dotted, gray] (K) -- (wx);
            \draw[loosely dotted, gray] (K) -- (wy);
            
            \draw[gray,dashed] (0, 1) ellipse (4.5cm and 2.5cm);
        \end{tikzpicture}
        \caption{The bowtie $(v, x_1y_1, x_2y_2)$ and vertices $w_x, w_y$}
        \label{fig:bowtie_lemma_proof}
    \end{figure}
    In what follows, we will repeatedly make use of the fact that every cycle of length 4 or 6 in $G' = G[N(K)]$ is balanced in the sense of Definition \ref{def: template} (because $G$ has no $K_r$-templates). We will use this for the 4-cycles $(v,x_1,w_x,x_2,v)$ and $(v,y_1,w_y,y_2,v)$, and the 6-cycle $(w_x,x_2,y_2,w_y,y_1,x_1,w_x)$. The fact that each of these cycles is balanced gives, respectively,
    \begin{equation}\label{eq:bowtie lemma balanced cycle 1}
    \{ f(vx_1),f(x_2w_x)\} = \{ f(vx_2),f(x_1w_x)\},
    \end{equation}
    \begin{equation}\label{eq:bowtie lemma balanced cycle 2}
    \{ f(vy_1),f(y_2w_y)\} = \{ f(vy_2),f(y_1w_y)\},
    \end{equation}
    and
    \begin{equation}\label{eq:bowtie lemma balanced cycle 3}
    \{ f(x_2w_x),f(y_2w_y),f(x_1y_1) \} = \{ f(x_1w_x), f(y_1w_y), f(x_2y_2) \}
    \end{equation}
    All equalities are as multisets. We now distinguish 3 cases:
    \begin{itemize}
        \item Suppose that $f(vx_1) = f(vx_2)$ and $f(vy_1) = f(vy_2)$. Then by 
        \eqref{eq:bowtie lemma balanced cycle 1} we have $f(x_1w_x) = f(x_2w_x)$ and by \eqref{eq:bowtie lemma balanced cycle 2} we have $f(y_1w_y) = f(y_2w_y)$. This allows us to cancel terms in \eqref{eq:bowtie lemma balanced cycle 3} and obtain
        $f(x_1y_1) = f(x_2y_2)$. Together with our assumptions in this case, this implies that \eqref{eq:bowtie lemma} holds. 
        \item Suppose that $f(vx_1) = f(vx_2)$ and $f(vy_1) \neq f(vy_2)$.
        Then \eqref{eq:bowtie lemma balanced cycle 1} gives
        $f(x_1w_x) = f(x_2w_x)$ while \eqref{eq:bowtie lemma balanced cycle 2} gives $f(vy_1) = f(y_1w_y)$ and $f(vy_2) = f(y_2w_y)$. Plugging these last two equalities into \eqref{eq:bowtie lemma balanced cycle 3} and canceling out $f(x_1w_x) = f(x_2w_x)$, we get
        $
        \{ f(vy_2), f(x_1y_1) \} = \{ f(vy_1), f(x_2y_2) \}
        $
        Together with the assumption $f(vx_1) = f(vx_2)$, this gives \eqref{eq:bowtie lemma} and establishes this case.
        Note that the case $f(vy_1) = f(vy_2)$ and $f(vx_1) \neq f(vx_2)$ is symmetric. 
        \item Finally, suppose that 
        $f(vx_1) \neq f(vx_2)$ and $f(vy_1) \neq f(vy_2)$. Then \eqref{eq:bowtie lemma balanced cycle 1} gives 
        $f(vx_1) = f(x_1w_x)$ and $f(vx_2) = f(x_2w_x)$, and similarly \eqref{eq:bowtie lemma balanced cycle 2} gives $f(vy_1) = f(y_1w_y)$ and $f(vy_2) = f(y_2w_y)$. Plugging these into \eqref{eq:bowtie lemma balanced cycle 3} gives \eqref{eq:bowtie lemma}.
    \end{itemize}
    \end{proof}

    \noindent
    The following lemma consists of some simple corollaries of Lemma \ref{lem: Bowtie fact}, which we record for future \nolinebreak use.
    
    \begin{lemma}\label{lem: Bowtie} Let $r \geq 3$ and let $(G, f)$ be an edge-colored graph satisfying $\delta(G) \ge \frac{r-1}{r}|G| + 3$ having no $K_r$-template. Let $K$ be a copy of $K_{r-2}$ in $G$, and let $(v, x_1y_1, x_2y_2)$ be a bowtie contained in $N(K)$. Then the following holds.
    \begin{enumerate}[(a)]
        \item If $f(vy_1), f(vx_1), f(y_1x_1)$ are all distinct, then 
        \[\{f(vx_1), f(vy_1)\}=\{f(vx_2), f(vy_2)\} \text{ and } f(x_1y_1) = f(x_2y_2).\]
        \item If $f(vx_1) = f(vy_1) \neq f(x_1y_1)$, then 
        \[f(vx_2) = f(vy_2) = f(vx_1) = f(vy_1) \text{ and } f(x_1y_1) = f(x_2y_2).\]
        \item If $S := \{f(vx_1), f(vy_1), f(x_1y_1)\}$ contains exactly two colors, then the color in $S$ with multiplicity 1 also belongs to $\{f(vx_2), f(vy_2), f(x_2y_2)\}$.
    \end{enumerate}
    \end{lemma}
    \noindent
    Figure \ref{fig:bowtie_lemma} illustrates the patterns described by Lemma \ref{lem: Bowtie}. 
    \begin{figure}[H]
        \centering
        \begin{tikzpicture}[scale=1.2, every node/.style={circle, draw, minimum size = 0.8cm}]
            \node (y11) at (-1,1.75) {$y_1$};   
            \node (x11) at (-1,0.25) {$x_1$};
            \node (v1) at (0,1) {$v$};
            \node (y12) at (1,1.75) {$y_2$};   
            \node (x12) at (1,0.25) {$x_2$};
            \node (y21) at (2.5,1.75) {$y_1$};   
            \node (x21) at (2.5,0.25) {$x_1$};
            \node (v2) at (3.5,1) {$v$};
            \node (y22) at (4.5,1.75) {$y_2$};   
            \node (x22) at (4.5,0.25) {$x_2$};
            \node (y41) at (6,1.75) {$y_1$};   
            \node (x41) at (6,0.25) {$x_1$};
            \node (v4) at (7, 1) {$v$};
            \node (y42) at (8,1.75) {$y_2$};   
            \node (x42) at (8,0.25) {$x_2$};

            \node[draw=none, style = rectangle] at (0,2.2) {Item (a)};
            \node[draw=none, style = rectangle] at (3.5,2.2) {Item (b)};
            \node[draw=none, style = rectangle] at (7,2.2) {Item (c)};
            
            \draw[thick, red] (v1) -- (y12);
            \draw[thick, green] (v1) -- (x12);
            \draw[thick, blue] (y11) -- (x11);
            \draw[thick, green] (v1) -- (x11);
            \draw[thick, red] (v1) -- (y11);
            \draw[thick, blue] (y12) -- (x12);
            \draw[thick, red] (v2) -- (y22);
            \draw[thick, red] (v2) -- (x22);
            \draw[thick, red] (v2) -- (x21);
            \draw[thick, red] (v2) -- (y21);
            \draw[thick, blue] (y21) -- (x21);
            \draw[thick, blue] (y22) -- (x22);

            \draw[thick, blue] (v4) -- (y42);
            \draw[thick, green] (v4) -- (x42);
            \draw[thick, red] (v4) -- (x41);
            \draw[thick, green] (v4) -- (y41);
            \draw[thick, red] (y41) -- (x41);
            \draw[thick, blue] (y42) -- (x42);
        \end{tikzpicture}
        \caption{Bowties corresponding to Items (a)-(c) of Lemma \ref{lem: Bowtie}}
        \label{fig:bowtie_lemma}
    \end{figure}
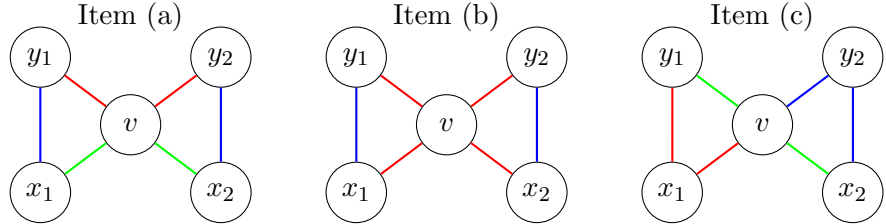

    \begin{proof}[Proof of Lemma \ref{lem: Bowtie}]
        For all items we use the equality \eqref{eq:bowtie lemma} from Lemma \ref{lem: Bowtie fact}.
        
        \noindent\textbf{Item (a).} Since $f(vx_1), f(vy_1) \neq f(x_1y_1)$ by assumption, the color $f(x_1y_1)$ on the LHS of \eqref{eq:bowtie lemma} must equal $f(x_2y_2)$ on the RHS. Canceling these terms yields $\{f(vx_1), f(vy_1)\} = \{f(vx_2), f(vy_2)\}$, as required.

        \noindent\textbf{Item (b).} Since $f(vx_1) = f(vy_1) \neq f(x_1y_1)$, we again get from \eqref{eq:bowtie lemma} that $f(x_1y_1) = f(x_2y_2)$. Canceling these terms yields $\{f(vx_2), f(vy_2)\} = \{f(vx_1), f(vy_1)\}$ (as multisets). Since the two elements on the RHS are the same, we conclude that $f(vx_2) = f(vy_2) = f(vx_1) = f(vy_1)$.
    
        \noindent\textbf{Item (c).} Let $c$ denote the color appearing exactly once in $S$. If $f(x_1y_1)=c$, then $f(vx_1) = f(vy_1) \neq c$, and the result follows from Item (b). Otherwise, 
        $c$ equals $f(vx_1)$ or $f(vy_1)$. Without loss of generality, $c = f(vx_1)$. Then
        $f(x_1y_1) = f(vy_1)$. Canceling these equal terms from opposite sides of \eqref{eq:bowtie lemma}, 
        we get $\{f(vx_2),f(vy_2)\} = \{f(x_2y_2),f(vx_1)\} = \{f(x_2y_2),c\}$.
        Hence, $c \in \{f(vx_2), f(vy_2)\}$.
    \end{proof}

	\subsection{Chains of $r$-cliques}\label{subsec:chains}
    In the proofs of Lemmas \ref{lem:main} and \ref{lem:aux_main} we will consider sequences of $r$-cliques such that each pair of consecutive $r$-cliques in the sequence shares many ($r-2$ or $r-1$) vertices. In some cases we will also require additional properties. In this section we establish the existence of these $r$-clique chains. We start with the following auxiliary lemma.
	\begin{lemma}\label{lem:chain aux}
    		Let $r \geq 2$ and let $G$ be a graph satisfying $\delta(G) > \frac{r-2}{r-1}|G|$. Let 
		$K_1 = \{x_1,\dots,x_{r-1}\}, K_2 = \{y_1,\dots,y_{r-1}\}$ be two $(r-1)$-cliques in $G$, and suppose that for every $1 \leq i \leq r-3$ and 
		$i+2 \leq j \leq r-1$, $y_ix_j \in E(G)$. Then there is a sequence $L_1,\dots,L_m$ of $(r-1)$-cliques such that $L_1 = K_1$, $L_m = K_2$, and $|L_i \cap L_{i+1}| \geq r-2$ for every $1 \leq i \leq m-1$.
	\end{lemma}
	\begin{proof}
		We will frequently use the fact that since $\delta(G) > \frac{r-2}{r-1}|G|$, every $r-1$ vertices in $G$ have a common neighbor. 
		
		The proof is by induction on $r$. For the base case $r=2$ we can simply take $m=2$ and 
        $L_1 := K_1 = \{x_1\}$, $L_2 := K_2 = \{y_1\}$. 
		Suppose now that $r \geq 3$. Let $z$ be a vertex adjacent to $y_1,\dots,y_{r-2}$ and $x_{r-1}$. We have the $(r-1)$-cliques 
        $M := \{y_1,\dots,y_{r-2},z\}$ and 
		$N := \{x_{r-1},y_1,\dots,y_{r-3},z\}$ (here we use that $x_{r-1}$ is adjacent to $y_1,\dots,y_{r-3}$ by the assumptions of the lemma). 
		Note that $|N \cap M|, |M \cap K_2| \geq r-2$. 
		Now, consider the two $(r-2)$-cliques $K'_1 := K_1 \setminus \{x_{r-1}\} = \{x_1,\dots,x_{r-2}\}$ and $K'_2 := N \setminus \{x_{r-1}\} = \{y_1,\dots,y_{r-3},z\}$. We have 
		$K'_1,K'_2 \subseteq N(x_{r-1})$. By Lemma \ref{lem:min degree in neighbourhood}, the graph $G' := G[N(x_{r-1})]$ satisfies 
        $\delta(G') > \frac{r-3}{r-2}|G'|$. We would like to apply the induction hypothesis to the $(r-2)$-cliques $K'_1,K'_2$ in the graph $G'$. This is possible because for every $1 \leq i \leq r-4$ and $i+2 \leq j \leq r-2$, $y_ix_j \in E(G')$. By the induction hypothesis, there is a sequence of $(r-2)$-cliques $L'_1,\dots,L'_m \subseteq N(x_{r-1})$ such that $L'_1 = K'_1$, $L'_m = K'_2$, and $|L'_i \cap L'_{i+1}| \geq r-3$ for every $1 \leq i \leq m-1$. Let $L_i := L'_i \cup \{x_{r-1}\}$ for $i \in [m]$. Then $L_1,\dots,L_m$ are $(r-1)$-cliques and $|L_i \cap L_{i+1}| \geq r-2$ for every $1 \leq i \leq m-1$. Also, $L_1 = K_1$ and 
        $L_m = K'_2 \cup \{x_{r-1}\} = N$. Now $K_1 = L_1,\dots,L_m = N, M, K_2$ is a sequence of $(r-1)$-cliques as required by the lemma. 
	\end{proof}

    Figure \ref{fig:clique chain} depicts the proof of Lemma \ref{lem:chain aux} for $r = 4$. The vertex $z$ is explicitly chosen to form the intermediate $3$-cliques, while $w$ arises from applying the induction hypothesis within $N(x_3)$ to connect the cliques $\{x_1,x_2\},\{y_1,z\}$. The sequentially numbered $3$-cliques denote the desired chain.
    \begin{figure}
        \centering
        \begin{tikzpicture}[scale=1.2, every node/.style={circle, draw, minimum size=0.8cm, thick}]
            \node (x1) at (1,{sqrt(3)}) {$x_1$};
            \node (x2) at (0,0) {$x_2$};
            \node (x3) at (2,0) {$x_3$};
            
            \node (y1) at (4,0) {$y_1$};
            \node (y2) at (5, {sqrt(3)}) {$y_2$};
            \node (y3) at (6, 0) {$y_3$};
    
            \node (w) at (2, 2.5) {$w$};
    
            \node (z) at (3, {sqrt(3)}) {$z$};
            
            \draw[thick] (x1) to (x2);
            \draw[thick] (x2) to (x3);
            \draw[thick] (x3) to (x1);
            \node[draw=none, inner sep=1pt] at (1,0.577) {1};
    
            \draw[thick, blue] (w) to (x1);
            \draw[thick, blue] (w) to (x3);
            \node[draw=none, inner sep=1pt] at (1.67, 1.41) {2};
    
            \draw[thick, blue] (w) to (z);
            \draw[thick, red] (z) to (x3);
            \node[draw=none, inner sep=1pt] at (2.33, 1.41) {3};
    
            \draw[thick] (x3) to (y1);
            \draw[thick, red] (z) to (y1);
            \node[draw=none, inner sep=1pt] at (3, 0.577) {4};
    
            \draw[thick] (y1) to (y2);
            \draw[thick, red] (z) to (y2);
            \node[draw=none, inner sep=1pt] at (4, 1.15) {5};
    
            \draw[thick] (y2) to (y3);
            \draw[thick] (y3) to (y1);
            \node[draw=none, inner sep=1pt] at (5, 0.577) {6};
        \end{tikzpicture}
        \caption{Chaining two $3$-cliques (illustration for the proof of Lemma \ref{lem:chain aux}).}
        \label{fig:clique chain}
    \end{figure}
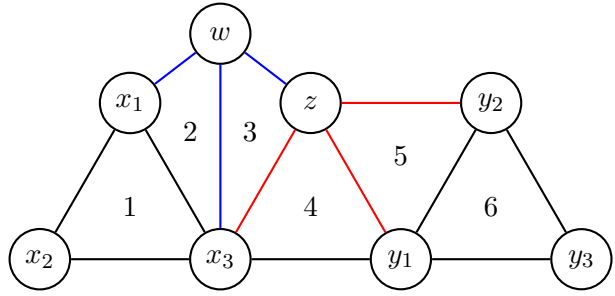
    
    In the following lemma, we obtain a chain of $r$-cliques connecting a given $r$-clique $L$ to a given edge $uw$. We can also guarantee that a given vertex $x \in L$ belongs to almost all $r$-cliques in the chain. This will be important later on.
	\begin{lemma}\label{lem:chain}
		Let $r \geq 2$ and let $G$ be a graph satisfying $\delta(G) > \frac{r-1}{r}|G|$.  
		Let $L \subseteq V(G)$ be an $r$-clique, let $x \in L$, and let $uw \in E(G)$. Then there is a sequence of $r$-cliques $L_1,\dots,L_m$ such that $L_1 = L$, $u,w \in L_m$, $x \in L_1,\dots,L_{m-2}$, and $|L_i \cap L_{i+1}| \geq r-1$ for every $1 \leq i \leq m-1$.  
	\end{lemma}
	\begin{proof}
		Similarly to before, we use the fact that any $r$ vertices have a common neighbor (since $\delta(G) > \frac{r-1}{r}|G|$).
		Write $L = \{x_1,\dots,x_r\}$ where $x_r = x$. 
		We define a sequence of vertices $y_1,\dots,y_{r-2}$ by induction as follows. Let $y_1$ be a common neighbor of $u,w,x_3,\dots,x_r$. Inductively, having chosen $y_1,\dots,y_{i-1}$, let $y_i$ be a common neighbor of $u,w,y_1,\dots,y_{i-1},x_{i+2},\dots,x_r$ (these are exactly $r$ vertices). This choice guarantees that $M := \{y_1,\dots,y_{r-2},u,w\}$ is an $r$-clique and that for every $1 \leq i \leq r-3$ and $i+2 \leq j \leq r-1$, $y_ix_j \in E(G)$ (note that this is exactly the condition in Lemma \ref{lem:chain aux}). Now, let $y_{r-1}$ be a common neighbor of $u,y_1,\dots,y_{r-2},x_r$. Then $N := \{y_1,\dots,y_{r-1},u\}$ and $P := \{x_r,y_1,\dots,y_{r-1}\}$ are $r$-cliques and $|P \cap N|, |N \cap M| \geq r-1$.
		
		We now apply Lemma \ref{lem:chain aux} to the $(r-1)$-cliques $K_1 := \{x_1,\dots,x_{r-1}\} = L \setminus \{x_r\}$ and 
        $K_2 := \{y_1,\dots,y_{r-1}\} = P \setminus \{x_r\}$ in the graph $G' := G[N(x_r)]$, noting that $K_1,K_2 \subseteq N(x_r)$. By Lemma \ref{lem:min degree in neighbourhood}, we have 
		$\delta(G') > \frac{r-2}{r-1}|G'|$. We already saw that the other condition of Lemma \ref{lem:chain aux} holds. Hence, by that lemma, $G'$ contains a sequence of $(r-1)$-cliques $L'_1,\dots,L'_k$ with $L'_1 = K_1$, $L'_k = K_2$, and $|L'_i \cap L'_{i+1}| \geq r-2$ for every $1 \leq i \leq k-1$. Let $L_i = L'_i \cup \{x_r\}$ for $1 \leq i \leq k$. Then $L_1,\dots,L_k$ are $r$-cliques, $L_1 = \{x_1,\dots,x_r\} = L$, $L_k = \{x_r,y_1,\dots,y_{r-1}\} = P$, $x = x_r \in L_1,\dots,L_k$, and $|L_i \cap L_{i+1}| \geq r-1$ for every $1 \leq i \leq k-1$. Now set $m := k+2$ and take $L_{m-1} := N$ and $L_m := M$. This gives the required sequence $L_1,\dots,L_m$.     
	\end{proof}
    
    Figure \ref{fig:clique to edge} illustrates the proof of Lemma \ref{lem:chain} for $r = 4$, connecting the initial $4$-clique $L = \{x_1, x_2, x_3, x_4\}$ to the target edge $uw$. Following the proof, we iteratively pick common neighbors $y_1, y_2$ (red edges) and $y_3$ (blue edges). Observe that $\{x_1, x_2, x_3\}$ and $\{y_1, y_2, y_3\}$ form $3$-cliques within the neighborhood $N(x_4)$. Because the edge $x_3y_1$ is present, the conditions of Lemma \ref{lem:chain aux} are satisfied, yielding a chain of $3$-cliques connecting $\{x_1,x_2,x_3\}$ and $\{y_1,y_2,y_3\}$ within $G[N(x_4)]$ (this chain is not depicted in the picture, but the reader is referred to Figure \ref{fig:clique chain}). Extending each $3$-clique in this chain by the vertex $x_4$ produces a chain of $4$-cliques from $L$ to $P = \{x_4, y_1, y_2, y_3\}$. The final sequence is obtained by appending $N = \{u, y_1, y_2, y_3\}$ and $M = \{u, w, y_1, y_2\}$.
    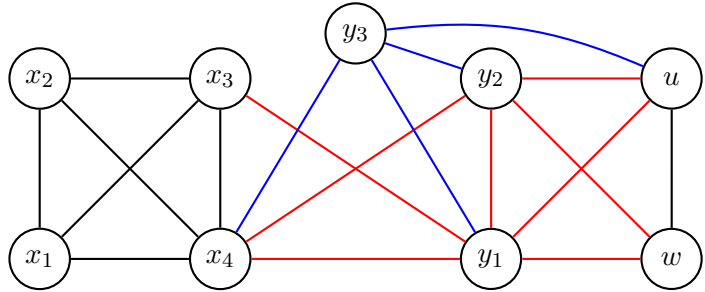
\begin{figure}[H]
        \centering
        \begin{tikzpicture}[scale=1.2, every node/.style={circle, draw, minimum size=0.8cm, thick}]
                \node (x1) at (0, 0) {$x_1$};
                \node (x2) at (0, 2) {$x_2$};
                \node (x3) at (2, 2) {$x_3$};
                \node (x4) at (2, 0) {$x_4$};
                
                \draw[thick] (x1) to (x2);
                \draw[thick] (x2) to (x3);
                \draw[thick] (x3) to (x1);
                \draw[thick] (x4) to (x3);
                \draw[thick] (x4) to (x2);
                \draw[thick] (x4) to (x1);

                \node (u) at (7, 2) {$u$};
                \node (w) at (7, 0) {$w$};

                \draw[thick] (u) to (w); 
                
                \node (y1) at (5, 0) {$y_1$};
                \node (y2) at (5, 2) {$y_2$};
                \node (y3) at (3.5, 2.5) {$y_3$};

                \draw[thick, red] (y1) to (u);
                \draw[thick, red] (y1) to (w);
                \draw[thick, red] (y1) to (x3);
                \draw[thick, red] (y1) to (x4);

                \draw[thick, red] (y2) to (u);
                \draw[thick, red] (y2) to (w);
                \draw[thick, red] (y2) to (y1);
                \draw[thick, red] (y2) to (x4);

                \draw[thick, blue, bend left=15] (y3) to (u);
                \draw[thick, blue] (y3) to (y1);
                \draw[thick, blue] (y3) to (y2);
                \draw[thick, blue] (y3) to (x4);
            \end{tikzpicture}
        \caption{Chaining a $4$-clique to an edge (illustration for the proof of Lemma \ref{lem:chain}).}
        \label{fig:clique to edge}
    \end{figure}
    
	Finally, we apply Lemma \ref{lem:chain} in the neighborhood of a vertex $v$ to obtain the following lemma. This is the lemma we will use in subsequent sections.
	\begin{lemma}[Chain lemma]\label{lem:chain main}
		Let $r \geq 3$, let $G$ be a graph satisfying $\delta(G) \geq \frac{r-1}{r}|G| + 2$, let $v \in V(G)$, let $K \subseteq N(v)$ be an $r$-clique, let $x \in K$ and let $uw \in E(G)$ with $u,w \in N(v)$. Then there is a sequence of $r$-cliques $K_1,\dots,K_m \subseteq N(v)$ such that $K_1 = K$, $u,w \in K_m$, $x \in K_1,\dots,K_{m-2}$, and $|K_i \cap K_{i+1}| = r-2$ for every $1 \leq i \leq m-1$.\footnote{We note that the requirement that $|K_i \cap K_{i+1}| = r-2$ (instead of $|K_i \cap K_{i+1}| \geq r-2$) is somewhat artificial, but it will be convenient for us that when applying Lemma \ref{lem:chain main}, we do not need to handle the case $|K_i \cap K_{i+1}| = r-1$.}
	\end{lemma}
    \begin{proof}
        If $u,w \in K$ then the assertion is trivial (by picking $m := 1$ and $K_1 := K$).
        So assume that $|K \cap \{u,w\}| \leq 1$, pick an arbitrary $y \in K \setminus \{x,u,w\}$, and let 
        $G' := G[N(v)] - y$. 
		We will apply Lemma \ref{lem:chain} in the graph $G'$ with parameter $r-1$ in place of $r$. 
        By Lemma \ref{lem:min degree in neighbourhood} we have $\delta(G[N(v)]) \geq \frac{r-2}{r-1}|N(v)| + 2$, and so $\delta(G') \geq \frac{r-2}{r-1}|N(v)| + 1 > \frac{r-2}{r-1}|G'|$. Thus, we may indeed apply Lemma \ref{lem:chain}.
        That lemma, with input $L := K \setminus \{y\}$, gives a sequence of $(r-1)$-cliques $L_1,\dots,L_m \subseteq V(G') \subseteq N(v)$ with $L_1 = L$, $u,w \in L_m$, $x \in L_1,\dots,L_{m-2}$, and $|L_i \cap L_{i+1}| \geq r-2$ for every $1 \leq i \leq m-1$. We may assume that 
		$|L_i \cap L_{i+1}| = r-2$ for all $i$, by omitting elements in the sequence $L_1,\dots,L_m$, if necessary (because if $|L_i \cap L_{i+1}| > r-2$ then $L_i = L_{i+1}$). 
		
		Next, we define $r$-cliques $K_1,\dots,K_m \subseteq N(v)$ as follows: First, $K_1 := L_1 \cup \{y\} = K$. 
        For $i=2,\dots,m$ in order, choose a vertex $z_i$ which is a common neighbor of $L_i \cup \{v\}$ and such that $z_i \notin K_{i-1} \cup L_{i+1}$, and put $K_i := L_i \cup \{z_i\}$. Such a vertex $z_i$ exists because the condition $\delta(G) \geq  \frac{r-1}{r}|G| + 2$ guarantees that any $r$ vertices have at least $2r$ common neighbors, while $|K_{i-1} \cup L_{i+1}| \leq 2r-1$. 
        This choice of $z_i$, and the fact that $y \notin L_1 \cup \dots \cup L_m$ by construction, guarantee that $|K_i \cap K_{i+1}| = |L_i \cap L_{i+1}| = r-2$ for all $1 \leq i \leq m-1$.
        Now we have a sequence of $r$-cliques $K_1,\dots,K_m \subseteq N(v)$ such that $K_1 = K$, $u,w \in K_m$, $x \in K_1,\dots,K_{m-2}$, and $|K_i \cap K_{i+1}| = r-2$ for every $1 \leq i \leq m-1$. 
	\end{proof}

	\subsection{Transferring colors}\label{subsec:transfering colors}
    Throughout this section, $r \geq 3$.
    In this section we further study the structure of edge-colored dense graphs with no $K_r$-templates. We show that given two $(r+1)$-cliques $M_1,M_2$ intersecting in $r-1$ vertices, in many cases one can deduce information on the colors in $M_2$ from the colors in $M_1$ (to this end we will use Lemmas \ref{lem: Bowtie fact} and \ref{lem: Bowtie}). It will be convenient to fix a vertex $v \in M_1 \cup M_2$ and to consider $K_i := M_i \setminus \{v\}$, so that $K_1,K_2 \subseteq N(v)$ are $r$-cliques and $|K_1 \cap K_2| = r-2$. Formally, our setting for Section \ref{subsec:transfering colors} is as follows:  
	
	\begin{setting}\label{setting}
	We are given a vertex $v \in V(G)$ and two $r$-cliques $K_1,K_2 \subseteq N(v)$.\footnote{So $K_i \cup \{v\}$ is an $(r+1)$-clique for $i=1,2$.} We assume that $|K_1 \cap K_2| = r-2$, and write $K_1 = \{ x_1,y_1,z_1,\dots,z_{r-2} \}$ and $K_2 = \{x_2,y_2,z_1,\dots,z_{r-2}\}$. 
    Finally, we assume that $\delta(G) \geq \frac{r-1}{r}n + 3$ and that $G$ has no $K_r$-template, enabling us to apply Lemmas \ref{lem: Bowtie fact} \nolinebreak and \nolinebreak \ref{lem: Bowtie}.
	\end{setting}
    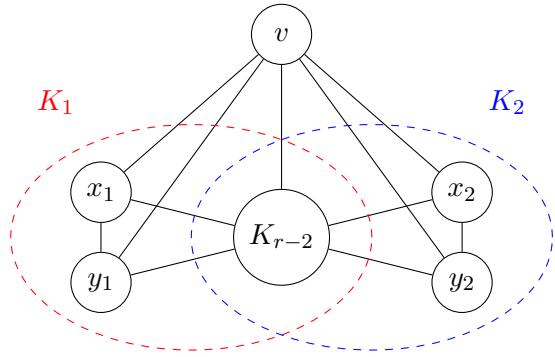
\begin{figure}[ht]
        \centering
        \begin{tikzpicture}[scale=1.2, every node/.style={circle, draw, minimum size=0.8cm}]
            \node (v) at (0, 2.25) {$v$};
            \node (z) at (0, 0) {$K_{r-2}$};
            
            \node (x1) at (-2, 0.5) {$x_1$};
            \node (y1) at (-2, -0.5) {$y_1$};
            
            \node (x2) at (2, 0.5) {$x_2$};
            \node (y2) at (2, -0.5) {$y_2$};
            
            \draw (v) -- (z);
            \draw (v) -- (x1);
            \draw (v) -- (y1);
            \draw (v) -- (x2);
            \draw (v) -- (y2);
            
            \draw (x1) -- (y1);
            \draw (x1) -- (z);
            \draw (y1) -- (z);
            
            \draw (x2) -- (y2);
            \draw (x2) -- (z);
            \draw (y2) -- (z);
            
            \draw[dashed, red] (-1, 0) ellipse (2cm and 1.25cm);
            \node[draw=none, text=red] at (-2.5, 1.5) {$K_1$};
            
            \draw[dashed, blue] (1, 0) ellipse (2cm and 1.25cm);
            \node[draw=none, text=blue] at (2.5, 1.5) {$K_2$};
            
        \end{tikzpicture}
        \caption{The triple $(v, K_1, K_2)$.}
        \label{fig:triple}
    \end{figure}
    \noindent
    See Figure \ref{fig:triple} for an illustration of this setting.
    It is convenient to introduce some more terminology:
    \begin{definition}[monochromatic vertex, centered clique] \label{def: mono}
		A vertex $v$ in a clique $M$ is {\em monochromatic in $M$} if all edges touching $v$ in $M$ have the same color. A {\em centered $(r+1)$-clique} is a pair $(v,K)$ where $K \subseteq N(v)$ is an $r$-clique (so $K \cup \{v\}$ is an $(r+1)$-clique) and all edges between $v$ and $K$ have the same color $c$ (so $v$ is monochromatic in $K \cup \{v\}$). The color $c$ is called the {\em main color} of $(v,K)$. 
	\end{definition}

    As mentioned above, our goal in this section is to consider situations where color-information can be transferred from $(v,K_1)$ to $(v,K_2)$. We therefore introduce the following definition. Recall that for a clique $K$ and a color $c$, we use $\#c(K)$ to denote the number of edges of color $c$ in $K$.
            	
	\begin{definition}[good/excellent triple]
		Let $v,K_1,K_2$ be as in Setting \ref{setting}. The triple $(v,K_1,K_2)$ is called {\em good} if for every color $c$, $\#c(K_1) = \#c(K_2)$. 
		
		If $(v,K_1)$ is a centered $(r+1)$-clique with main color $c_1$, then $(v,K_1,K_2)$ is called {\em excellent} if it is good and additionally, all edges between $v$ and $K_2$ have color $c_1$. 
	\end{definition}
    
	In what follows, we give sufficient conditions for being good/excellent. In the proofs, we will consider the bowties $(v,x_1y_1,x_2y_2)$ and $(z_i,x_1y_1,x_2y_2)$ for $1 \leq i \leq r-2$ (the notation is as in Setting \ref{setting}) and apply Lemmas \ref{lem: Bowtie fact} and \ref{lem: Bowtie}, using the assumption that there are no $K_r$-templates. Note that we may apply these lemmas because each of the above bowties is contained in the common neighborhood of an $(r-2)$-clique. Indeed, $(v,x_1y_1,x_2y_2) \subseteq N(z_1,\dots,z_{r-2})$, and $(z_i,x_1y_1,x_2y_2)$ is contained in the common neighborhood of $\{v\} \cup \{z_j : j \in [r-2]\setminus \{i\}\}$. 
	
	\begin{lemma}\label{lem:transferring colors basic}
		In Setting \ref{setting}, if $f(x_1y_1) = f(x_2y_2)$ then $(v,K_1,K_2)$ is good. In particular, 
		if there is $1 \leq i \leq r-2$ such that the triangle $z_ix_1y_1$ is rainbow, then $(v,K_1,K_2)$ is good. 
	\end{lemma}
	\begin{proof}
		Fix any $1 \leq i \leq r-2$, and consider the bowtie $(z_i,x_1y_1,x_2y_2)$. As $f(x_1y_1) = f(x_2y_2)$, Lemma \ref{lem: Bowtie fact} (i.e., the equality \eqref{eq:bowtie lemma}) gives 
        $\{f(z_ix_2),f(z_iy_2)\} = \{f(z_ix_1),f(z_iy_1)\}$ as a multiset. This implies that every color $c$ appears the same number of times in $K_2$ as in $K_1$. Indeed, this follows from $K_1 \cap K_2 = \{z_1,\dots,z_{r-2}\}$, $f(x_2y_2) = f(x_1y_1)$, and $\{f(z_ix_2),f(z_iy_2)\} = \{f(z_ix_1),f(z_iy_1)\}$ for every $1 \leq i \leq r-2$. 
		So $(v,K_1,K_2)$ is good, as required. 
		
		The ``In particular" part follows from the first part, because if $z_ix_1y_1$ is rainbow then $f(x_1y_1) = f(x_2y_2)$ by Lemma \ref{lem: Bowtie}(a) applied to the bowtie $(z_i,x_1y_1,x_2y_2)$. 
	\end{proof}
	\begin{lemma}\label{lem:transferring colors monochromatic vertex}
		In Setting \ref{setting}, suppose that $(v,K_1)$ is a centered $(r+1)$-clique with main color $c_1$. If $(v,K_1,K_2)$ is {\bf not} excellent, then all of the following are satisfied:
		\begin{enumerate}
			\item $f(x_1y_1) = c_1$.
			\item $c_2 := f(x_2y_2) \neq c_1$.
			\item For every $1 \leq i \leq r-2$, (at least) one of the edges $z_ix_1,z_iy_1$ has color $c_1$, and (at least) one of the edges $z_ix_2,z_iy_2$ has color $c_2$.
			\item Either $K_2 \cup \{v\}$ has no monochromatic vertex, or it has a unique monochromatic vertex $w$. In the latter case, $w \in \{x_2,y_2\}$, $w$ is monochromatic in color $c_2$, and all edges between $v$ and $K_2 \setminus \{w\}$ have color $c_1$.  
		\end{enumerate}
		Moreover, in any case (i.e., even if $(v,K_1,K_2)$ is not good, which means that 1-4 hold), every color appearing in $K_2$, except possibly the color $c_2$, also appears in $K_1$. 
	\end{lemma}
	\begin{proof}
		We first assume that $(v,K_1,K_2)$ is not excellent and prove that Items 1-4 hold. 
		For Item 1, suppose by contradiction that $f(x_1y_1) \neq c_1$. 
		Consider the bowtie $(v,x_1y_1,x_2y_2)$. 
        By assumption, $f(vx_1) = f(vy_1) = c_1$.
        By Lemma \ref{lem: Bowtie}(b) we get that $f(x_2y_2) = f(x_1y_1)$ and $f(vx_2) = f(vy_2) = c_1$. 
		Then by Lemma \ref{lem:transferring colors basic}, $(v,K_1,K_2)$ is good. Also, we just showed that all edges between $v$ and $K_2 = \{x_2,y_2,z_1,\dots,z_{r-2}\}$ have color $c_1$ (for $z_1,\dots,z_{r-2}$ this holds by assumption), so $(v,K_1,K_2)$ is excellent, a contradiction. This proves Item 1. 
		
		Now we prove Item 2. We already proved that $f(x_1y_1) = c_1$ (Item 1). Suppose by contradiction that $f(x_2y_2) = c_1 = f(x_1y_1)$. 
		Then by Lemma \ref{lem:transferring colors basic}, $(v,K_1,K_2)$ is good.
		Also, by considering the bowtie $(v,x_1y_1,x_2y_2)$, we get by Lemma \ref{lem: Bowtie fact} (i.e., the equality \eqref{eq:bowtie lemma}) that 
		$f(vx_2) = f(vy_2) = c_1$. Hence, $(v,K_1,K_2)$ is excellent, again giving a contradiction. 
		 
		Now we prove Item 3. Fix any $1 \leq i \leq r-2$ and consider the bowtie 
		$(z_i,x_1y_1,x_2y_2)$.
		As $f(x_1y_1) = c_1 \neq c_2 = f(x_2y_2)$, by the equality \eqref{eq:bowtie lemma} in Lemma \ref{lem: Bowtie fact} we have that one of the edges $z_ix_1,z_iy_1$ has color $c_1$, and one of the edges $z_ix_2,z_iy_2$ has color $c_2$. 
		
		Next, we prove Item 4. 
		First, let us consider the bowtie $(v,x_1y_1,x_2y_2)$. Similarly to the above, the fact that
        $f(x_1y_1) = c_1 \neq c_2 = f(x_2y_2)$ and equality \eqref{eq:bowtie lemma} imply that one of the edges $vx_2,vy_2$ has color $c_2$. 
		Now suppose that $K_2 \cup \{v\}$ has a monochromatic vertex $w$ (else Item 4 clearly holds). Note that $w \neq v$, because $v$ touches both colors $c_1$ and $c_2$ (indeed, the edges between $v$ and $\{z_1,\dots,z_{r-2}\}$ have color $c_1$, and one of the edges $vx_2,vy_2$ has color $c_2$). Also, $w \neq z_i$ for all $1 \leq i \leq r-2$, because $f(vz_i) = c_1$ and $z_i$ touches color $c_2$ in $K_2$ by Item 3. So $w \in \{x_2,y_2\}$. As $f(x_2y_2) = c_2$, $w$ is monochromatic in color $c_2$. Suppose without loss of generality that $w = x_2$. It remains to show that all edges between $v$ and $\{z_1,\dots,z_{r-2},y_2\}$ have color $c_1$. This is true for the edges $vz_i$, $1 \leq i \leq r-2$. Considering the bowtie $(v,x_1y_1,x_2y_2)$ again, and using that $f(vx_2) = f(x_2y_2) = c_2$ and that $f(vx_1) = f(vy_1) = f(x_1y_1) = c_1$, we get from equality \eqref{eq:bowtie lemma} that $f(vy_2) = c_1$, as required. 
		
		Finally, we prove the second part. 
        If $(v,K_1,K_2)$ is good then by definition, every color appearing in $K_2$ also appears in $K_1$, so we are done. Otherwise, Items 1-4 hold, and we need
        to show that except possibly for $c_2$, every color appearing in $K_2$ also appears in $K_1$. So fix any edge $e \in E(K_2)$ with $f(e) \neq c_2$, and let us show that $f(e)$ appears in $K_1$. If $e \subseteq \{z_1,\dots,z_{r-2}\}$ then this clearly holds (as 
        $\{z_1,\dots,z_{r-2}\} \subseteq K_1$). Also, $e \neq x_2y_2$ because $f(x_2y_2) = c_2$. So $e \in \{z_ix_2,z_iy_2\}$ for some $1 \leq i \leq r-2$. Consider the bowtie $(z_i,x_1y_1,x_2y_2)$, and note that $c_2 \in \{f(z_ix_2), f(z_iy_2)\}$ (by Item 3). By Item 2, $f(x_2y_2)$ also has color $c_2$.
        Since $f(e) \neq c_2$, Lemma \ref{lem: Bowtie}(c) implies that $f(e) \in \{f(x_1y_1), f(z_ix_1), f(z_iy_1)\}$. So $f(e)$ appears in $K_1$, as required.
	\end{proof}

    \noindent
    In what follows, we will also need the following notion and lemma.
	\begin{definition}[Isolated vertex]\label{def:isolated vertex}
		Let $(v,K)$ be a centered $(r+1)$-clique with main color $c$. 
		A vertex $t \in K$ is {\em isolated} if no edge of $K$ touching $t$ has color $c$. 
	\end{definition}
	
	\begin{lemma}\label{lem:isolated vertex}
		In Setting \ref{setting}, suppose that $(v,K_1)$ is a centered $(r+1)$-clique with main color $c$. If $(v,K_1,K_2)$ is excellent and $t \in \{z_1,\dots,z_{r-2}\}$ is isolated in $(v,K_1)$, then $t$ is also isolated in $(v,K_2)$. 
	\end{lemma}
	\begin{proof}
		As $(v,K_1,K_2)$ is excellent, $(v,K_2)$ is also a centered $(r+1)$-clique with main color $c$.
		Suppose by contradiction that $t = z_i$ is not isolated in $(v,K_2)$. Then one of the edges $z_ix_2,z_iy_2$ has color $c$. By considering the bowtie $(v,x_1y_1,x_2y_2)$ and using that $f(vx_1) = f(vy_1) = f(vx_2) = f(vy_2) = c$, we get from the equality \eqref{eq:bowtie lemma} in Lemma \ref{lem: Bowtie fact} that $f(x_1y_1) = f(x_2y_2)$. 
        Now, considering the bowtie $(z_i,x_1y_1,x_2y_2)$, we get by Lemma \ref{lem: Bowtie fact} that 
		$\{f(z_ix_1),f(z_iy_1)\} = \{f(z_ix_2),f(z_iy_2)\}$, and hence one of the edges $z_ix_1,z_iy_1$ has color $c$, a contradiction to $z_i = t$ being isolated in $(v,K_1)$.
	\end{proof}
	
	\subsection{A vertex whose neighborhood contains only few colors}\label{subsec:U = N(v)}

	We now use the results of the previous sections to argue that in certain situations, there exists a vertex $v$ whose neighborhood contains at most $\binom{r}{2}$ different colors. 
    This is a key step towards proving Lemma \ref{lem:main}.
    Our proof strategy is as follows: Start with a carefully-chosen $r$-clique $K \subseteq N(v)$. Given an edge $uw$ with $u,w \in N(v)$, take the sequence $K = K_1,\dots,K_m \subseteq N(v)$ of $r$-cliques given by Lemma \ref{lem:chain main}, and argue that every color appearing in $K_m$ (and so in particular the color $f(uw)$) also appears in $K_1$. This suffices because $K_1$ clearly contains at most $\binom{r}{2}$ different colors. The argument proceeds by showing that $(v,K_i,K_{i+1})$ is good/excellent, using Lemmas \ref{lem:transferring colors basic} and \ref{lem:transferring colors monochromatic vertex}.\footnote{We note that there is one case (in the proof of Lemma \ref{lem:K_{r+1} with monochromatic vertex}) where we could not argue that $(v,K_{m-1},K_m)$ is good and instead use a slightly different argument.} 
    
	We prove three statements of the type described above, namely Lemmas \ref{lem:very simple K_{r+1}}, \ref{lem:K_{r+1} with many colors} and \ref{lem:K_{r+1} with monochromatic vertex}. 
	\begin{lemma}\label{lem:very simple K_{r+1}}
		Let $r \geq 3$. Let $(G, f)$ be an edge-colored graph with $\delta(G) \geq \frac{r-1}{r}|G| + 3$ and no $K_r$-template. Suppose that $G$ contains a centered $(r+1)$-clique $(v,K)$ with main color $c_1$ such that no edge inside $K$ has color $c_1$. Then every color in $N(v)$ appears in $K$. In particular, $N(v)$ contains at most $\binom{r}{2}$ colors.
	\end{lemma}
	\begin{proof}
		Fix any edge $uw \in E(G)$ with $u,w \in N(v)$. Apply Lemma \ref{lem:chain main} to obtain a sequence of $r$-cliques $K_1,\dots,K_m \subseteq N(v)$ with $K_1 = K$, $u,w \in K_m$, and $|K_i \cap K_{i+1}| = r-2$ for every $1 \leq i \leq m-1$. (The vertex $x$ given as input to Lemma \ref{lem:chain main} does not matter here, and we take it to be an arbitrary vertex of $K$.)
		
		We prove by induction on $i$ that for every $1 \leq i \leq m$, $(v,K_i)$ is a centered $(r+1)$-clique with main color $c_1$, and $\#c(K_{i}) = \#c(K)$ for every color $c$. The base case clearly holds (as $K_1 = K$). So let $1 \leq i \leq m-1$, suppose that the claim holds for $i$, and let us prove it for $i+1$. By the induction hypothesis and the assumption of the lemma, we have $\#c_1(K_i) = \#c_1(K) = 0$. It now follows from Lemma \ref{lem:transferring colors monochromatic vertex} that $(v,K_i,K_{i+1})$ is excellent. Indeed, if not, then by Item 1 of Lemma \ref{lem:transferring colors monochromatic vertex} (applied to $(v,K_i,K_{i+1})$) we have that $\#c_1(K_i) \geq 1$, a contradiction. As $(v,K_i,K_{i+1})$ is excellent, we have that $(v,K_{i+1})$ is a centered $(r+1)$-clique with main color $c_1$, and $\#c(K_{i+1}) = \#c(K_i) = \#c(K)$ for every color $c$, establishing the induction step. 
		
		Finally, we see that every color appearing in $K_m$, and in particular the color $f(uw)$, also appears in $K$. This proves the lemma.
	\end{proof}
    \noindent
    For the next lemma, we need the following simple fact:
    \begin{lemma}\label{lem:many colors in r-clique}
        Let $r \geq 3$ and let $M$ be an edge-colored $(r+1)$-clique containing at least $\binom{r}{2}+1$ different colors. Then there is an $r$-clique $K \subseteq M$ which contains at least $\binom{r-1}{2}+2$ different colors.
    \end{lemma}
    \begin{proof}
        The proof is by induction on $r$. For the base case $r=3$, write $M = \{x,y,z,w\}$. Suppose first that all edges touching $w$ have the same color. Then either $xyz$ is a rainbow triangle (and we are done), or $M$ contains at most 3 colors, a contradiction. Suppose now that $xw,yw$ have the same color, say 1, and $zw$ has a different color, say 2. If $xz$ or $yz$ has a color different from 1 and 2, then we found a rainbow triangle. Otherwise, both $xz$ and $yz$ have color 1 or 2, so $M$ has at most 3 colors (namely, the colors $1,2$ and the color of $xy$), a contradiction. Finally, suppose that $xw,yw,zw$ all have different colors. In order to have no rainbow triangle, each edge inside $x,y,z$ needs to have the same color as one of the edges $xw,yw,zw$. This again means that $M$ has only 3 colors, a contradiction.

        For the induction step, suppose that $r \geq 4$.
        If $M$ is rainbow then every $r$-clique $K \subseteq M$ clearly contains $\binom{r}{2} \geq \binom{r-1}{2}+2$ colors (this inequality holds for $r \geq 3$). Otherwise, some color is repeated. Let $v$ be a vertex touching a repeated color, and let 
        $K = M \setminus \{v\}$. Let $a$ be the number of different colors contained in $K$. We claim that $a \geq \binom{r-1}{2}+1$. Indeed, since $v$ touches a repeated color, we have that either $v$ touches two edges of the same color, or some edge touching $v$ has the same color as an edge contained in $K$. In either case, we lose at most $r-1$ colors when deleting $v$, so $a \geq \binom{r}{2} + 1 - (r-1) = \binom{r-1}{2}+1$. If $a \geq \binom{r-1}{2}+2$ then we are done, so suppose that 
        $a = \binom{r-1}{2}+1$. Since $M$ contains at least $\binom{r}{2}+1$ different colors, there are at least $\binom{r}{2}+1-a = r-1$ colors which are present on edges touching $v$ but not in $K$. Hence, there are edges $e_1,\dots,e_{r-1}$ touching $v$ such that $e_1,\dots,e_{r-1}$ have pairwise-distinct colors and these colors do not appear in $K$. By the induction hypothesis (which applies as $a = \binom{r-1}{2}+1$), there is an $(r-1)$-clique $L \subseteq K$ containing edges of at least $\binom{r-2}{2}+2$ different colors. Adding $v$ to $L$, we add at least $r-2$ of the edges $e_1,\dots,e_{r-1}$, and hence at least $r-2$ new colors. So altogether, $L \cup \{v\}$ contains at least $\binom{r-2}{2}+2 + (r-2) = \binom{r-1}{2}+2$ different colors, as required.
    \end{proof}
    
	\begin{lemma}\label{lem:K_{r+1} with many colors}
		Let $r \geq 3$. Let $(G, f)$ be an edge-colored graph with $\delta(G) \geq \frac{r-1}{r}|G| + 3$. Suppose that $G$ has no $K_r$-template, and that $G$ has an $(r+1)$-clique $M$ which contains at least $\binom{r}{2}+1$ different colors. 
		Then there is $v \in V(G)$ such that $N(v)$ contains at most $\binom{r}{2}$ different colors. 
	\end{lemma}
	\begin{proof}
		By Lemma \ref{lem:many colors in r-clique}, there is $v \in M$ such that $K := M-v$ contains at least $\binom{r-1}{2}+2$ different colors. 
		We now proceed similarly to the proof of Lemma \ref{lem:very simple K_{r+1}}. 
		Fix any edge $uw \in E(G)$ with $u,w \in N(v)$. 
		Apply Lemma \ref{lem:chain main} to obtain a sequence of $r$-cliques $K_1,\dots,K_m \subseteq N(v)$ with $K_1 = K$, $u,w \in K_m$, and $|K_i \cap K_{i+1}| = r-2$ for every $1 \leq i \leq m-1$. (The vertex $x$ given as input to Lemma \ref{lem:chain main} does not matter here, and we take it to be an arbitrary vertex of $K$.)
		
		We prove by induction that for every $1 \leq i \leq m$, $\#c(K_i) = \#c(K)$ for every color $c$. This holds for $i=1$ as $K_1 = K$. Now let $1 \leq i \leq m-1$, suppose the claim holds for $i$, and let us prove it for $i+1$. Write $K_i = \{x_1,y_1,z_1,\dots,z_{r-2}\}$, $K_{i+1} = \{x_2,y_2,z_1,\dots,z_{r-2}\}$ (as in Setting \ref{setting}). We claim that there is $1 \leq j \leq r-2$ such that the triangle $z_jx_1y_1$ is rainbow. Indeed, if not, then for every $1 \leq j \leq r-2$, either $f(z_jx_1) = f(z_jy_1)$, or one of the edges $z_jx_1,z_jy_1$ has the same color as $x_1y_1$. In any case, this means that there are edges $e_1,\dots,e_{r-2}$ in $K_i$ such that each $e_j$ ($1 \leq j \leq r-2$) has the same color as some edge in $E(K_i) \setminus \{e_1,\dots,e_{r-2}\}$ (here $e_j$ is $z_jx_1$ or $z_jy_1$). Hence, $K_i$ contains at most $\binom{r}{2} - (r-2) = \binom{r-1}{2}+1$ different colors. By the induction hypothesis, 
		$\#c(K_i) = \#c(K)$ for every color $c$. Hence, $K$ contains at most $\binom{r-1}{2}+1$ different colors, a contradiction to the choice of $K$. This proves our claim that there is $1 \leq j \leq r-2$ such that $z_jx_1y_1$ is a  rainbow triangle. Now, by Lemma \ref{lem:transferring colors basic}, $(v,K_i,K_{i+1})$ is good. Hence, for every color $c$, we have $\#c(K_{i+1}) = \#c(K_i) = \#c(K)$, establishing the induction \nolinebreak step. 
		
		Finally, we see that every color appearing in $K_m$, and in particular the color $f(uw)$, also appears in $K$. Since $uw$ was an arbitrary edge in $N(v)$, we get that $N(v)$ contains at most $\binom{r}{2}$ different \nolinebreak colors. 
	\end{proof}

	\begin{lemma}\label{lem:K_{r+1} with monochromatic vertex}
		Let $r \geq 3$. Let $(G, f)$ be an edge-colored graph with $\delta(G) \geq \frac{r-1}{r}|G| + 3$. Suppose that $G$ has no $K_r$-template and that every copy of $K_{r+1}$ has a monochromatic vertex. Then there is $v \in V(G)$ such that $N(v)$ contains at most $\binom{r}{2}$ different colors. 
	\end{lemma}
	\begin{proof}
		The minimum degree condition implies that $G$ contains an $(r+1)$-clique, and the assumption of the lemma states that every $(r+1)$-clique in $G$ has a monochromatic vertex. 
		We choose a centered $(r+1)$-clique $(v,K)$ in $G$ according to the following rules (recall Definition \ref{def:isolated vertex}).
		\begin{enumerate}
			\item[(a)] If there is a centered $(r+1)$-clique $(v,K)$ which has an isolated vertex, then choose such a $(v,K)$ for which the main color of $(v,K)$ appears the least possible number of times inside $K$.
			\item[(b)] Else (i.e., if no centered $(r+1)$-clique $(v,K)$ in $G$ has an isolated vertex), then choose $(v,K)$ to be an arbitrary centered $(r+1)$-clique. 
		\end{enumerate}
		Choose a vertex $t \in K$ as follows. If $(v,K)$ has an isolated vertex, then let $t$ be such a vertex, and else let $t \in K$ be arbitrary. We fix this choice of $(v,K)$ and $t$ for the rest of the proof. 
		
		We now proceed similarly to the previous two proofs. 
		Our goal is to show that $N(v)$ contains at most $\binom{r}{2}$ colors. So fix any edge $uw \in E(G)$ with $u,w \in N(v)$. We apply Lemma \ref{lem:chain main} with input $x := t$.
		Lemma \ref{lem:chain main} gives a sequence of $r$-cliques $K_1,\dots,K_m \subseteq N(v)$ such that $K_1 = K$, $u,w \in K_m$, $t \in K_1,\dots,K_{m-2}$ and $|K_i \cap K_{i+1}| = r-2$ for every $1 \leq i \leq m-1$. 
		
		
		We now continue via an inductive argument, as in the previous two proofs, going over the chain $K_1,\dots,K_m$. Note that unlike previous proofs, here we first only handle $K_1,\dots,K_{m-1}$ (in Claim \ref{claim:K_1,...,K_{m-1}}), and then handle $K_m$ separately (in Claim \ref{claim:K_m}). 
		Let $c_1$ be the main color $(v,K)$. 
		\begin{claim}\label{claim:K_1,...,K_{m-1}}
			{\color{white} text}
			\begin{enumerate}
				\item For every $1 \leq i \leq m-1$, all edges between $v$ and $K_i$ have color $c_1$. 
				\item If $t$ is isolated in $(v,K)$, then for every $1 \leq i \leq m-2$, $t$ is isolated in $(v,K_i)$. 
				\item For every $1 \leq i \leq m-1$, 
				$\#c(K_i) = \#c(K)$ for every color $c$. 
			\end{enumerate}
		\end{claim}
		\begin{proof}
			We prove all three items by induction on $i$, simultaneously. 
			In the base case $i=1$, Items 1-3 are immediate because $K_1 = K$. 
			Now fix any $1 \leq i \leq m-2$, and suppose that Items 1-3 hold for $i$. We prove that they hold for $i+1$. 
			
			Write $K_i = \{x_1,y_1,z_1,\dots,z_{r-2}\}$, $K_{i+1} = \{x_2,y_2,z_1,\dots,z_{r-2}\}$ (as in Setting \ref{setting}). 
			Note that $t \in K_i$ (as $i \leq m-2$), and that if $i \leq m-3$ then 
			$t \in K_i \cap K_{i+1} = \{z_1,\dots,z_{r-2}\}$.  
			If $(v,K_i,K_{i+1})$ is excellent, then, by definition, all edges between $v$ and $K_{i+1}$ have color $c_1$ and 
			$\#c(K_{i+1}) = \#c(K_i) = \#c(K)$ for every color $c$. Furthermore, if $i \leq m-3$ and $t$ is isolated in $(v,K)$, then by the induction hypothesis $t$ is isolated in $(v,K_i)$, and then by Lemma \ref{lem:isolated vertex}, $t$ is isolated in $(v,K_{i+1})$. 
			Hence, if $(v,K_i,K_{i+1})$ is excellent then Items 1-3 hold for $i+1$. 
			
			So suppose that $(v,K_i,K_{i+1})$ is not excellent. We will now use Lemma \ref{lem:transferring colors monochromatic vertex} and our choice of $(v,K)$ in Items (a)-(b) to arrive at a contradiction. By Items 1 and 3 of Lemma \ref{lem:transferring colors monochromatic vertex}, every vertex in $K_i$ touches an edge of color $c_1$. Hence, $(v,K_i)$ has no isolated vertex. 
			As $t \in K_i$, it follows that $t$ is not isolated in $(v,K)$ (because else $t$ would be isolated in $(v,K_i)$ by the induction hypothesis). 
			By the choice of $t$, we get that $(v,K)$ has no isolated vertex. 
			Now, by the choice of $(v,K)$ in Items (a)-(b), no centered $(r+1)$-clique $(v,K)$ in $G$ has an isolated vertex. However, let us consider the $(r+1)$-clique $M := K_{i+1} \cup \{v\}$. By the assumptions of Lemma \ref{lem:K_{r+1} with monochromatic vertex}, every $(r+1)$-clique in $G$ has a monochromatic vertex. Hence, by Item 4 of Lemma \ref{lem:transferring colors monochromatic vertex}, $M$ has a monochromatic vertex $w \in \{x_2,y_2\}$. 
			In fact, setting $K' := (K_{i+1} \setminus \{w\}) \cup \{v\}$, Items 2 and 4 of Lemma \ref{lem:transferring colors monochromatic vertex} imply that $(w,K')$ is a centered $(r+1)$-clique with main color $c_2 \neq c_1$, and that $v$ is isolated in 
			$(w,K')$. Thus, $G$ does contain a centered $(r+1)$-clique which has an isolated vertex, namely, $(w,K')$. This is a contradiction, completing the induction step.   
		\end{proof}
		\begin{claim}\label{claim:K_m}
			Every color appearing in $K_m$ also appears in $K_{m-1}$. 
		\end{claim}
		\begin{proof}
			As before, write $K_{m-1} = \{x_1,y_1,z_1,\dots,z_{r-2}\}$, $K_m = \{x_2,y_2,z_1,\dots,z_{r-2}\}$. 
			By Item 1 of Claim \ref{claim:K_1,...,K_{m-1}}, $(v,K_{m-1})$ is a centered $(r+1)$-clique with main color $c_1$. 
			If $(v,K_{m-1},K_m)$ is excellent then we are done, so from now on assume that $(v,K_{m-1},K_m)$ is not excellent. 
			Then by Items 2 and 4 of Lemma \ref{lem:transferring colors monochromatic vertex}, the following holds: $c_2 := f(x_2y_2) \neq c_1$; 
			$K_m \cup \{v\}$ has a unique monochromatic vertex $w \in \{x_2,y_2\}$; setting $K' := (K_m \setminus \{w\}) \cup \{v\}$, we have that $(w,K')$ is a centered $(r+1)$-clique with main color $c_2$ and $v$ is isolated in 
			$(w,K')$. (Here we again use that $K_m \cup \{v\}$ has a monochromatic vertex due to the assumption of Lemma \ref{lem:K_{r+1} with monochromatic vertex} that every $(r+1)$-clique in $G$ has a monochromatic vertex.) 
			Without loss of generality, $w = x_2$, and so $K' = \{y_2,z_1,\dots,z_{r-2},v\}$.
			 
			By the second part of Lemma \ref{lem:transferring colors monochromatic vertex}, every color in $K_m$ also appears in $K_{m-1}$, except possibly for the color $c_2$. So it remains to show that $c_2$ appears in $K_{m-1}$. Suppose otherwise. We count the number of appearances of $c_1$ in $K_{m-1}$ and of $c_2$ in $K'$. By Items 1 and 3 of Lemma \ref{lem:transferring colors monochromatic vertex}, $c_1$ appears at least $1 + (r-2) = r-1$ times in $K_{m-1}$. On the other hand, no edge inside $\{z_1,\dots,z_{r-2}\}$ has color $c_2$ (because we assumed that $c_2$ does not appear in $K_{m-1}$), and no edge going between $v$ and $K' \setminus \{v\} = \{y_2,z_1,\dots,z_{r-2}\}$ has color $c_2$ (because $v$ is isolated in $(x_2,K')$). Hence, in $K'$, only edges between $y_2$ and $\{z_1,\dots,z_{r-2}\}$ can have color $c_2$. Therefore, $\#c_2(K') \leq r-2$. 
			We conclude that $\#c_1(K_{m-1}) \geq r-1 > r-2 \geq \#c_2(K')$. Also, by Item 3 of Claim \ref{claim:K_1,...,K_{m-1}}, we have $\#c_1(K_{m-1}) = \#c_1(K)$, so $\#c_1(K) > \#c_2(K')$. 
			Summarizing, we see that $(x_2,K')$ is a centered $(r+1)$-clique with main color $c_2$, having an isolated vertex (namely $v$), and satisfying 
			$\#c_2(K') < \#c_1(K)$. 
			But this contradicts the choice of $(v,K)$ in Item (a) above. 
		\end{proof}
	
		We now complete the proof of the lemma. By combining Claim \ref{claim:K_m} with Item 3 of Claim \ref{claim:K_1,...,K_{m-1}}, we see that every color appearing in $K_m$ also appears in $K$. In particular, $f(uw)$ appears in $K$. As $uw$ was an arbitrary edge in $N(v)$, it follows that $N(v)$ contains at most $\binom{r}{2}$ colors. 
	\end{proof}

	\subsection{$(r+1)$-cliques with no monochromatic vertex}\label{subsec:U derived from K}
    
    In this section, we consider another way of obtaining a set $U$ as in Lemma \ref{lem:main}; namely, we take $U$ to be the set of all vertices which have $r$ neighbors in some fixed $(r+1)$-clique $K$. We need the following definition.
    Throughout this section, $r \ge 3$. 
    \begin{definition}[2-colored-star]\label{def: star}A clique $K$ is a {\em 2-colored-star} if there are distinct colors $c$ and $c'$ and a vertex $v$, called the {\em center}, such that all edges incident to $v$ have color $c$, and all edges inside $K \setminus \{v\}$ have color $c'$.
    \end{definition}
    \noindent
    The following lemma appears in \cite{BCPT:21} (for the case of $q=2$ colors). For completeness, we give a proof. 
    \begin{lemma}
        Let $r \geq 3$. Let $(G, f)$ be an edge-colored graph with no $K_r$-template and $K$ be a copy of $K_{r + 2}$ in $G$. Then $K$ is monochromatic or a 2-colored star.
        \label{lem:very simple kr+2}
    \end{lemma}
    \begin{proof}
        Since $|V(K)| = r+2 \ge 5$, any four vertices in $K$ lie in the common neighborhood of a copy of $K_{r-2}$. Thus, since $G$ is assumed to contain no $K_r$-templates, every 4-cycle in $K$ is balanced.
        
        Suppose $K$ is not monochromatic. Then there exists a vertex $v_1 \in V(K)$ incident to edges $v_1u_1$ and $v_1u_2$ with distinct colors $c_1$ and $c_2$, respectively. Let $v_2 \in V(K) \setminus \{v_1, u_1, u_2\}$. The balanced 4-cycle 
        $(v_1, u_1, v_2, u_2, v_1)$ implies 
        $f(v_2 u_1) = c_1$ and $f(v_2 u_2) = c_2$.
        
        Consider the 4-cycle $(u_1, u_2, v_1, v_2, u_1)$. 
        Since this cycle is balanced, we have 
        $\{u_1 u_2, v_1 v_2\} = \{c_1, c_2\}$. Without loss of generality, assume $f(u_1 u_2) = c_1$ and $f(v_1 v_2) = c_2$ (the other case is symmetrical by switching $c_1,c_2$ and $u_1,u_2$).
        
        We claim that $u_1$ is the center of a $c_1$-star and $K \setminus \{u_1\}$ is monochromatic in $c_2$. Let $w \in K \setminus \{v_1, u_1, u_2\}$. The balanced cycle 
        $(v_1, u_1, w, u_2, v_1)$ implies $\{c_1, f(u_2w)\} = \{f(u_1 w), c_2\}$. Since $c_1 \neq c_2$, we must have $f(u_1 w) = c_1$ and $f(u_2 w) = c_2$.
        
        Finally, we show that all remaining edges in $K \setminus \{u_1\}$ are colored $c_2$. 
        We have already shown this for edges incident to $u_2$. Fix any two distinct 
        $w,w' \in K \setminus \{u_1,u_2\}$. The cycle 
        $(u_1,u_2,w,w',u_1)$ is balanced and we have $f(u_1u_2) = f(u_1w') = c_1$ and $f(u_2w) = c_2$. Hence, $f(ww') = c_2$, as \nolinebreak required.
    \end{proof}

    \noindent
    Next, we need the following simple lemma.
    \begin{lemma}\label{lem:clique vertex copy}
        Let $(G, f)$ be an edge-colored graph with no $K_r$-templates. Let $K = \{x_1, \dots, x_{r + 1}\}$ be a $K_{r + 1}$ and suppose that $x_{r+1}$ is not monochromatic in $K$. Let $y \in V(G)$ be adjacent to $x_1, \dots, x_r$. Then for all $i \in [r]$,
            \[f(x_i y) = f(x_i x_{r + 1}).\]
    \end{lemma}
    \begin{proof}
        Let $i \in [r]$. Since $x_{r + 1}$ is not a monochromatic vertex, there exists $j \in [r]$ for which
        \[a := f(x_i x_{r + 1}) \neq f(x_j x_{r + 1}) =: b.\]
        The 4-cycle $(y, x_i, x_{r + 1}, x_j,y)$ lies in the neighborhood of the $(r - 2)$-clique 
        $\{x_1, \dots, x_r\} \setminus \{x_i, x_j\}$. 
        Hence, as $(G,f)$ has no $K_r$-template, this 4-cycle is balanced. Therefore,
        \[\{f(x_i y), b\} = \{f(x_j y), a\}.\]
        As $a \neq b$, $f(x_i y) = a = f(x_i x_{r + 1})$ as required.
    \end{proof}
    
    The following lemma is the main result of Section \ref{subsec:U derived from K}. It shows that under certain conditions, one can obtain a set $U$ satisfying the requirements of Lemma \ref{lem:main}.
	\begin{lemma} \label{lem: no mono few colors}
		Let $r \geq 3$. Let $(G, f)$ be an edge-colored graph with $\delta(G) \geq \frac{r-1}{r}|G|$ and no $K_r$-template. Suppose that $G$ has an $(r+1)$-clique $K = \{x_1,\dots,x_{r+1}\}$ with no monochromatic vertex and containing at most $\binom{r}{2}$ different colors, and assume that $G$ has no $(r+2)$-clique containing $K$. Then there is a subset $U \subseteq V(G)$ with $|U| \geq \delta(G)$ such that $U$ contains at most $\binom{r}{2}$ different colors. 
	\end{lemma}
    \begin{proof}
        Let $U := \{v \in V(G) \mid v \text{ has $r$ neighbors in $K$}\}$. 
        First, we show that $U$ is of the desired size. By double counting, we have
        \begin{equation}\label{eq:no mono few colors double counting}
        \sum_{x \in K}d(x) = \sum_{u \in V(G)}d_K(u).
        \end{equation}
        For a lower bound, 
        \begin{equation}\label{eq:no mono few colors lower}
        \sum_{x \in K} d(x) \geq (r+1) \cdot \delta(G).
        \end{equation}
        Since there are no $(r + 2)$-cliques containing $K$, $d_{K}(u) \leq r$ for every $u$. So for an upper bound,
        \begin{equation}\label{eq:no mono few colors upper}
        \sum_{u \in V(G)}d_K(u) \leq r \cdot \vert U \vert + (r - 1)(\vert G \vert - \vert U \vert) = \vert U \vert + (r - 1) \vert G \vert.
        \end{equation}
        Combining \eqref{eq:no mono few colors double counting}, \eqref{eq:no mono few colors lower} and \eqref{eq:no mono few colors upper}, we get 
        $$
        |U| \geq (r+1) \cdot \delta(G) - (r-1)|G| \geq \delta(G),
        $$
        using that $\delta(G) \geq \frac{r-1}{r}|G|$.

        Next, we show that the colors of the edges in $U$ come from the colors of the edges in $K$, as follows. For each $i \in [r + 1]$, let
        \[U_i := \{u \in U \; \vert \; ux_j \in E(G) \text{ for all } j \neq i\}.\]
        In particular, $U = \bigcup_{i = 1}^{r + 1}U_i$, and $U_1,\dots,U_{r+1}$ are pairwise-disjoint.
        Since there is no monochromatic vertex in $K$, Lemma \ref{lem:clique vertex copy} applies: for all $u_i \in U_i$ and $j \in [r+1] \setminus \{i\}$,
        \begin{equation}\label{eq:no mono few colors u_i,x_i}
            f(u_i x_j) = f(x_i x_j).
        \end{equation}
        \begin{claim}\label{claim:no mono few colors 1}
            For every $1 \leq i < j \leq r + 1$ and every $u_iu_j \in E[U_i, U_j]$, it holds that 
            $f(u_i u_j) = f(x_i x_j)$.
        \end{claim}
        \begin{proof}
            Put $K' := (K \setminus \{x_i\}) \cup \{u_i\}$.
            By \eqref{eq:no mono few colors u_i,x_i}, 
            $f(u_i x_j) = f(x_i x_j)$ for all $j \in [r+1] \setminus \{i\}$.
            Hence, $K'$ is color-isomorphic to $K$; more precisely, the map from $K'$ to $K$ which maps $u_i$ to $x_i$ and fixed all other vertices preserves all colors. In particular, $K'$ has no monochromatic vertex (because $K$ has no such vertex). Now, by applying Lemma \ref{lem:clique vertex copy} to $K'$ and $y := u_j$, we get $f(u_iu_j) = f(u_ix_j) = f(x_ix_j)$, as required.
        \end{proof}
        \begin{claim}\label{claim:no mono few colors 2}
            For each $i \in [r + 1]$, $U_i$ is an independent set.
        \end{claim}
        \begin{proof}
            Suppose that there exist $v, w \in U_i$ with $vw \in E(G)$. Then 
            $$
            M := \{ v, w, x_1, \dots, x_{i - 1}, x_{i + 1}, \dots, x_{r + 1}\}
            $$
            is an $(r+2)$-clique. Since $G$ has no $K_r$-template, by Lemma \ref{lem:very simple kr+2}, $M$ is monochromatic or a 2-colored-star. Note that the map from $M \setminus \{w\}$ to $K$ which maps $v$ to $x_i$ and fixes all $x_j, j\neq i$, is color-preserving (by \eqref{eq:no mono few colors u_i,x_i}). Hence, if $M$ is monochromatic or a 2-colored-star then so is $K$. But this means that $K$ has a monochromatic vertex, a contradiction to the assumption of the lemma.
            %
        \end{proof}
        By Claim \ref{claim:no mono few colors 2}, $E(U) = \bigcup_{1 \leq i < j \leq r+1}E[U_i,U_j]$, and by Claim \ref{claim:no mono few colors 1}, all edges in $E[U_i,U_j]$ have color $f(x_ix_j)$. 
        Thus, every edge of $U$ is colored with one of the $\binom{r}{2}$ colors used in $K$.
    \end{proof}
	\subsection{Putting it all together: proof of Lemmas \ref{lem:main} and \ref{lem:aux_main}}
	\begin{proof}[Proof of Lemma \ref{lem:main}]
        Suppose first that $G$ contains a non-monochromatic $(r+2)$-clique $M$. By Lemma \ref{lem:very simple kr+2}, $M$ is a 2-colored-star (since $G$ has no $K_r$-template). Hence, $G$ contains a centered $(r+1)$-clique $(v,K)$ whose main color does not appear inside $K$ (indeed, take $v$ to be the center of $M$ and take $K$ to be any subset of $M \setminus \{v\}$ of size $r$). Now, by Lemma \ref{lem:very simple K_{r+1}}, the conclusion of Lemma \ref{lem:main} holds for $U := N(v)$. Thus, we may assume that all $(r+2)$-cliques in $G$ are monochromatic. Next, we split our analysis into the following three cases:
        \begin{enumerate}
            \item There exists a copy of $K_{r + 1}$ with no monochromatic vertex and $\leq \binom{r}{2}$ colors.
            \item There exists a copy of $K_{r +1}$ with $\ge \binom{r}{2} + 1$ colors.
            \item Every copy of $K_{r + 1}$ has a monochromatic vertex.
        \end{enumerate}
        Note that (at least) one of Cases 1-3 must hold. Indeed, as $\delta(G) > \frac{r-1}{r}|G|$, Tur\'an's theorem implies that $G$ contains a copy of $K_{r+1}$. Now, if Cases 1-2 do not hold, then Case 3 holds.
        In Case 1, let $K$ be a copy of $K_{r+1}$ as in Item 1. Then $K$ cannot be completed into an $(r+2)$-clique since we assumed that all $(r+2)$-cliques in $G$ are monochromatic (while $K$ is clearly not monochromatic). Therefore, the hypothesis of Lemma \ref{lem: no mono few colors} holds, implying the conclusion of Lemma \ref{lem:main}.

        In Cases 2 and 3, we may apply Lemmas \ref{lem:K_{r+1} with many colors} and \ref{lem:K_{r+1} with monochromatic vertex}, respectively. In both cases, we see that the conclusion of Lemma \ref{lem:main} holds with $U = N(v)$ for some vertex $v$.
	\end{proof}

    \begin{proof}[Proof of Lemma \ref{lem:aux_main}]
        Since $\delta(G) > \frac{r}{r + 1}|G|$, Tur\'an's theorem implies that $G$ contains a copy of $K_{r + 2}$. By Lemma \ref{lem:very simple kr+2}, every $K_{r + 2}$ is either monochromatic or a 2-colored-star (since $G$ has no $K_r$-template).
    
        \medskip
        \noindent \textbf{Case 1: $G$ contains a $K_{r+2}$-copy forming a 2-colored-star.} \\
        Let $K$ be such a $K_{r+2}$-copy, where the center $v$ of $K$ is incident to edges of color $c_1$, and the remaining clique $K' = K \setminus \{v\}$ is monochromatic in color $c_2$. Lemma \ref{lem:very simple K_{r+1}} implies that every color appearing in $N(v)$ must also appear in $K'$. (To be precise, we may apply Lemma \ref{lem:very simple K_{r+1}} with the centered $(r+1)$-clique $(v,K'')$, where $K''$ is any subset of $K'$ of size $r$). Since $K'$ contains only edges of color $c_2$, the set $U := N(v)$ only contains edges of color $c_2$. Furthermore, $|U| \ge \delta(G)$, satisfying the requirements.
    
        \medskip
        \noindent \textbf{Case 2: Every $K_{r + 2}$-copy in $G$ is monochromatic.} \\
        Let $K$ be a monochromatic $K_{r + 2}$ in color $c$. Fix a vertex $v \in K$ and let $U = N(v)$. Clearly, $|U| \ge \delta(G)$. To show that $G[U]$ is monochromatic, consider an arbitrary edge $uw \in E(G[U])$. We apply Lemma \ref{lem:chain main} (with parameter $r + 1$) to obtain a sequence of $(r + 1)$-cliques $K_1, \dots, K_m \subseteq U$ such that $K_1 = K \setminus \{v\}$, $u, w \in K_m$, and $|K_i \cap K_{i + 1}| = r - 1$ for all $1 \leq i \leq m - 1$.
        
        For each $i \in [m]$, the set $K_i \cup \{v\}$ forms a $K_{r + 2}$ in $G$. By our assumption, $K_i \cup \{v\}$ is monochromatic. Furthermore, consecutive cliques $K_i \cup \{v\}$ and $K_{i + 1} \cup \{v\}$ share $r \ge 3$ vertices, and thus share at least one edge. This forces the entire sequence of cliques to be monochromatic in the same color $c$. In particular, $K_m$ is monochromatic in color $c$, implying $f(uw) = c$. This proves that $U$ only contains edges of color $c$.
    \end{proof}
    \section{Proofs of the Main Theorems}\label{sec:proof of theorems}
    In this section, we prove Theorems \ref{thm:general upper bound}-\ref{thm:divisibility condition}. 
    The proofs of the upper bounds in these theorems proceed by first applying the cleaning lemma (Lemma \ref{lem:template cleaning}) to reduce to the case where $(G,f)$ has no $K_r$-templates, then applying Lemma \ref{lem:main} or \ref{lem:aux_main} to find a big set $U$ which contains only few colors, and finally applying Lemma \ref{lem: extremal tiling main} to find a $K_r$-tiling with large discrepancy. All upper bound proofs follow this plan, with the only difference being the value of $q$ and the assumed minimum degree $\delta(G)$. To avoid repetition, we execute these steps in the following two lemmas. 
    \begin{lemma} \label{lem:lower_bound_color}
    Let $r \ge 3$ and $q \ge 2$. For every $\varepsilon > 0$, there exists $\zeta > 0$ such that for every $n$ with $n \gg 1/\epsilon$ and $r \mid n$, the following holds. Let $(G,f)$ be a $q$-edge-colored $n$-vertex graph with $\delta(G) \geq \left(\frac{r-1}{r}+\epsilon\right)n$. Then there exists a $K_r$-tiling $\mathcal{T}$ of $G$ and a color $c \in [q]$ which appears on at least a
    \begin{equation}\label{eq:lower_bound_color}
    \min\left( \frac{1}{q} + \zeta, \; \frac{\rho - \varepsilon}{\binom{r}{2}} \right) 
    \end{equation}
    fraction of the edges of $\mathcal{T}$, where $\rho = \frac{2\delta(G)}{n} - 1$.
    \end{lemma}
    \begin{proof}
        We first apply Lemma \ref{lem:template cleaning} to $(G, f)$ with parameter $\xi := \varepsilon/3$. 
        If Item 1 of Lemma \ref{lem:template cleaning} holds then we are done. Otherwise, Lemma \ref{lem:template cleaning} provides a $K_r$-template-free subgraph $G' \subseteq G$ such that $|V(G')| \ge (1 - \xi)n$ and $\delta(G') \ge \delta(G) - \xi n$, and such that the set $V(G) \setminus V(G')$ admits a $K_r$-tiling $\mathcal{T}_{\mathrm{rem}}$. Note that
        \[
        \delta(G') \ge \left(\frac{r - 1}{r} + \varepsilon - \xi\right)n \ge \left(\frac{r - 1}{r} + \frac{\varepsilon}{2}\right)n.
        \]
        Thus, by Theorem \ref{thm:hajnal sze}, $G'$ admits a perfect $K_r$-tiling $\mathcal{T}_{G'}$. Let $\mathcal{T} = \mathcal{T}_{G'} \cup \mathcal{T}_{\mathrm{rem}}$.
    
        Next, we apply Lemma \ref{lem:main} to $G'$ to obtain a set $U \subseteq V(G')$ with $|U| \ge \delta(G') \geq \delta(G) - \xi n$ such that $G'[U]$ contains edges of at most $\binom{r}{2}$ colors. Defining $\rho = 2\delta(G)/n - 1$, Lemma \ref{lem: extremal tiling main} implies that the fraction of edges in $\mathcal{T}_{G'}$ induced by $U$ is at least
        \[
        \rho' := \frac{2 |U|}{|G'|} - 1 \ge \frac{2 (\delta(G) - \xi n)}{n} - 1 = \rho - 2\xi.
        \]
        (Note that Lemma \ref{lem: extremal tiling main} applies to any $K_r$-tiling of $K_n$, and hence also any $K_r$-tiling of $G$.)
        By the pigeonhole principle, there exists a color 
        $c \in [q]$ that appears on at least a 
        $\frac{\rho - 2\xi}{\binom{r}{2}}$-fraction of the edges of $\mathcal{T}_{G'}$. Since $\mathcal{T}_{G'}$ contains at least a $(1 - \xi)$-fraction of all edges in $\mathcal{T}$, the color $c$ appears on at least the following fraction of the edges of $\mathcal{T}$:
        \[
        \frac{\rho - 2\xi}{\binom{r}{2}}(1 - \xi) 
        = \frac{\rho - 2\xi - \rho\xi + 2\xi^2}{\binom{r}{2}} 
        \ge \frac{\rho - 3\xi}{\binom{r}{2}} \ge \frac{\rho - \varepsilon}{\binom{r}{2}}.
        \]
        The penultimate inequality holds because $\rho \le 1$, and the final inequality because $3\xi = \varepsilon$.
    \end{proof}

    We also need the following variant of Lemma \ref{lem:lower_bound_color} in which the applicant of Lemma \ref{lem:main} is replaced with Lemma \ref{lem:aux_main}. This is relevant in the regime $\delta(G) > \frac{r}{r+1}n$.
    
    \begin{lemma} \label{lem:aux_lowerbound_color}
    Let $r \ge 3$ and $q \ge 2$. For every $\varepsilon > 0$, there exists $\zeta > 0$ such that for every $n$ with $n \gg 1/\epsilon$ and $r \mid n$, the following holds. Let $(G,f)$ be a $q$-edge-colored $n$-vertex graph with $\delta(G) \geq \left(\frac{r}{r + 1}+\varepsilon\right)n$. Then there exists a $K_r$-tiling $\mathcal{T}$ of $G$ and a color $c \in [q]$ which appears on at least a
    \begin{equation}\label{eq:aux_lowerbound_color}
    \min\left( \frac{1}{q} + \zeta, \; \rho - \varepsilon \right) 
    \end{equation}
    fraction of the edges of $\mathcal{T}$, where $\rho = \frac{2\delta(G)}{n} - 1$.
    \end{lemma}
    The proof of Lemma \ref{lem:aux_lowerbound_color} is essentially identical to that of Lemma \ref{lem:main}, with the difference being that we apply Lemma \ref{lem:aux_main} in place of Lemma \ref{lem:main}. (We may apply Lemma \ref{lem:aux_main} due to the minimum degree assumption in Lemma \ref{lem:aux_lowerbound_color}.)
    Thus, instead of obtaining a set $U$ which contains at most $\binom{r}{2}$ colors (as guaranteed by Lemma \ref{lem:main}), we obtain a set $U$ which is monochromatic (as guaranteed by Lemma \ref{lem:aux_main}). Hence, in \eqref{eq:aux_lowerbound_color} we do not have the $\frac{1}{\binom{r}{2}}$ factor appearing in \eqref{eq:lower_bound_color}. We omit the proof of Lemma \ref{lem:aux_lowerbound_color}.

    We are now ready to prove Theorems \ref{thm:general upper bound}-\ref{thm:divisibility condition}.
    \begin{proof}[Proof of Theorem \ref{thm:general upper bound}]
        Let $\varepsilon > 0$ and let $n$ be an integer with $n \gg \varepsilon$ and $r \mid n$. Let $G$ be an $n$-vertex graph with $\delta(G) \ge (\frac{r}{r + 1} + \varepsilon)n$, and let $f \colon E(G) \to [q]$ be an edge coloring.
    
        By Lemma \ref{lem:aux_lowerbound_color}, there exists a $K_r$-tiling $\mathcal{T}$ of $G$ and a color $c \in [q]$ which appears on at least a $\mu$-fraction of the edges of $\mathcal{T}$, where
        \[
            \mu := \min\left( \frac{1}{q} + \zeta, \; \rho - \epsilon \right) 
        \]
        and $\rho = \frac{2\delta(G)}{n} - 1$. To lower bound the second term, we have
        \[
            \rho - \varepsilon = \left(\frac{2\delta(G)}{n} - 1\right) - \varepsilon \ge \left(\frac{2r}{r+1} + 2\varepsilon - 1\right) - \varepsilon = \frac{r-1}{r+1} + \varepsilon.
        \]
        For $r \ge 3$, we have $\frac{r-1}{r+1} \ge \frac{1}{2}$. Since $q \ge 2$, we have $\frac{1}{q} \le \frac{1}{2}$. Thus,
        \[
            \mu \ge \min\left(\frac{1}{q} + \zeta, \; \frac{1}{2} + \varepsilon\right) \ge \frac{1}{q} + \min(\zeta, \varepsilon).
        \]
        It follows that $\mathcal{T}$ has discrepancy $\Omega(\zeta n)$. This proves that $\delta_{r,q} \le \frac{r}{r+1}$.
    \end{proof}

        \begin{proof}[Proof of Theorems \ref{thm:r=4} and \ref{thm:r=3}]
        We prove both theorems simultaneously. Suppose first that $\binom{r}{2} \le q \le \binom{r + 1}{2}$. In this case, we have $\delta_{r,q} = \frac{r}{r+1}$, with the lower bound given by Lemma \ref{lem: thm 1.6 & 1.7 lo part 1} and the upper bound by Theorem \ref{thm:general upper bound}. This establishes the first case of Theorem \ref{thm:r=4}, and the case $3 \leq q \leq 6$ of Theorem \ref{thm:r=3}. In the case $r=3,q=2$ we again have 
        $\delta_{3,2} = \frac{r}{r+1} = \frac{3}{4}$, with the lower bound given by Lemma \ref{lem:thm 1.8 lo} and the upper bound by Theorem \ref{thm:general upper bound}.
        
        Suppose from now on that $q > \binom{r + 1}{2}$. 
        We first establish the lower bounds. For $q = \binom{r + 1}{2} + 1$ (this applies both to $r\geq 4$ and $r=3$), the lower bound $\delta_{r,q} \geq \frac{r^2+1}{r^2+r+2}$ (so $\delta_{3,7} \geq \frac{5}{7}$) is given by Lemma \ref{lem: thm 1.6 & 1.7 lo part 2}; indeed, substituting 
        $q = \binom{r + 1}{2} + 1 = (r^2+r+2)/2$ into $\frac{1}{2} + \frac{r(r-1)}{4q}$ gives $\delta_{r, q} \ge \frac{r^2+1}{r^2+r+2}$. 
        For $r = 3$ and $q = 8$, the lower bound $\delta_{3,8} \geq \frac{11}{16}$ again follows from Lemma \ref{lem: thm 1.6 & 1.7 lo part 2}, because substituting $r=3,q=8$ into $\frac{1}{2} + \frac{r(r-1)}{4q}$ gives $\frac{11}{16}$.
        Finally, note that we have 
        $\delta_{r, q} \geq \frac{r-1}{r}$ for all $r,q$, because there exist graphs $G$ with 
        $\delta(G) = \frac{r-1}{r}n - 1$ and no $K_r$-tiling (so in particular no $K_r$-tiling with high discrepancy). 
        The above covers all cases of Theorems \ref{thm:r=4} and \ref{thm:r=3}.

        Next, we prove the upper bounds. 
        In each case, let $\delta^*$ denote the claimed threshold for given $r,q$. Let $\varepsilon > 0$, and let $(G, f)$ be a $q$-edge-colored $n$-vertex graph with $n \gg 1/\varepsilon$, $r \mid n$, and $\delta(G) \ge (\delta^* + \varepsilon)n$. Since $\delta^* \ge \frac{r-1}{r}$ in all cases, we may apply Lemma \ref{lem:lower_bound_color} to find a $K_r$-tiling $\mathcal{T}$ and a color $c \in [q]$ appearing on at least a fraction $\mu$ of the edges, where
        \[
            \mu := \min\left(\frac{1}{q} + \zeta, \; \frac{\rho - \varepsilon}{\binom{r}{2}}\right) \quad \text{and} \quad \rho = \frac{2\delta(G)}{n} - 1.
        \]
        To prove that $\mathcal{T}$ has discrepancy $\Omega(n)$, it suffices to show that $\mu - \frac{1}{q} = \Omega(1)$ (where the $\Omega(1)$ term depends on $\varepsilon,\zeta$). The first entry in the minimum is $\frac{1}{q}+\zeta$, so it suffices to consider the second entry. Substituting $\rho \ge (2\delta^* - 1) + 2\varepsilon$, we obtain:
        \[
            \frac{\rho - \varepsilon}{\binom{r}{2}} \ge \frac{(2\delta^* - 1) + \varepsilon}{\binom{r}{2}} = \frac{2\delta^* - 1}{\binom{r}{2}} + \frac{\varepsilon}{\binom{r}{2}}.
        \]
        Therefore, it suffices to verify that $\frac{2\delta^* - 1}{\binom{r}{2}} \ge \frac{1}{q}$. We check this for each regime in Theorems \ref{thm:r=4} and \ref{thm:r=3}, as follows:
    
        \begin{itemize}
            \item If $q = \binom{r+1}{2} + 1$, the threshold is $\delta^* = \frac{r^2 + 1}{r^2 + r + 2}$ both in Theorem \ref{thm:r=4} and in Theorem \ref{thm:r=3}. Then:
            \[
                \frac{2\delta^* - 1}{\binom{r}{2}} = \frac{\frac{2(r^2 + 1)}{r^2 + r + 2} - 1}{\frac{r(r-1)}{2}} = \frac{2}{r^2 + r + 2} = \frac{1}{q}.
            \]
            \item If $r=3$ and $q=8$, the threshold is $\delta^* = \frac{11}{16}$, and we have 
            $$
            \frac{2\delta^* - 1}{\binom{r}{2}} = \frac{6/16}{3} = \frac{1}{8} = \frac{1}{q}.
            $$
            \item Suppose that $r \geq 4$ and $q \ge \binom{r+1}{2} + 2$, 
            or $r = 3$ and $q \geq 9$. In both cases,
            the threshold is $\delta^* = \frac{r-1}{r}$. Then:
            \[
                \frac{2\delta^* - 1}{\binom{r}{2}} = \frac{\frac{2(r-1)}{r} - 1}{\frac{r(r-1)}{2}} = \frac{2(r-2)}{r^2(r-1)}.
            \]
            It is not hard to check that 
            $\frac{2(r-2)}{r^2(r-1)} \geq \frac{1}{q}$ if $r \geq 4$ and $q \ge \binom{r+1}{2} + 2$, or if 
            $r=3$ and $q \geq 9$. Thus, in both cases, the required inequality holds. 
        \end{itemize}
        This completes the proof of Theorems \ref{thm:r=4} and \ref{thm:r=3}.
    \end{proof}
    \begin{proof}[Proof of Theorem \ref{thm:divisibility condition}]
        The lower bound is Lemma \ref{lem:thm 1.8 lo} and the upper bound is Theorem \ref{thm:general upper bound}.
    \end{proof}

    
\bibliographystyle{abbrv}
\bibliography{library}

\end{document}